\documentclass{amsart}%
\usepackage{amssymb}
\usepackage{amsmath}
\usepackage{hyperref}
\usepackage[pagewise]{lineno}
\usepackage[american]{babel}
\usepackage{amsfonts}
\usepackage{graphicx}%
\newtheorem{theorem}{Theorem}[section]
\newtheorem{lemma}[theorem]{Lemma}
\newtheorem{corollary}[theorem]{Corollary}
\theoremstyle{definition}
\newtheorem{definition}[theorem]{Definition}
\newtheorem{example}[theorem]{Example}
\newtheorem{proposition}[theorem]{Proposition}

\theoremstyle{remark}
\newtheorem{remark}[theorem]{Remark}
\numberwithin{equation}{section}

\begin{document}
\title{Invariance entropy, quasi-stationary measures and control sets}
\author{Fritz Colonius}
\email{fritz.colonius@math.uni-augsburg.de}

\begin{abstract}
Institut f\"{u}r Mathematik, Universit\"{a}t Augsburg, Augsburg, Germany

\end{abstract}
\begin{abstract}
For control systems in discrete time, this paper discusses measure-theoretic
invariance entropy for a subset $Q$ of the state space with respect to a
quasi-stationary measure obtained by endowing the control range with a
probability measure. The main results show that this entropy is invariant
under measurable transformations and that it is already determined by certain
subsets of $Q$ which are characterized by controllability properties.

\end{abstract}
\date{\today }
\maketitle

\section{Introduction}

Metric invariance entropy provides a measure-theoretic analogue of the
topological notion of (feedback) invariance entropy $h_{inv}(Q)$ of
deterministic control systems, cf. Nair, Evans, Mareels and Moran
\cite{NEMM04} and Kawan \cite{Kawa13}. The present paper discusses metric
invariance entropy and its relations to controllability properties. We
consider control systems in discrete time of the form%
\begin{equation}
x_{k+1}=f(x_{k},u_{k}),k\in\mathbb{N}=\{0,1,\ldots\}, \label{0.1}%
\end{equation}
where $f:M\times\Omega\rightarrow M$ is continuous and $M$ and $\Omega$ are
metric spaces.

For an initial value $x_{0}\in M$ at time $k=0$ and control $u=(u_{k}%
)_{k\geq0}\in\mathcal{U}:=\Omega^{\mathbb{N}}$ we denote the solutions by
$x_{k}=\varphi(k,x_{0},u),k\in\mathbb{N}$. The notion of invariance entropy
$h_{inv}(Q)$ describes the average data rate needed to keep the system in a
given subset $Q$ of $M$ (forward in time). It is constructed with some analogy
to topological entropy of dynamical systems. A major difference of entropy in
a control context to entropy for dynamical systems (cf. Walters \cite{Walt82}
or Viana and Oliveira \cite{VianO16}) is that the minimal required entropy for
the considered control task is of interest instead of the \textquotedblleft
total\textquotedblright\ entropy generated by the dynamical system, hence the
infimum over open covers or partitions is taken instead of the supremum.

The present paper discusses notions of metric invariance entropy modifying and
extending the analysis in Colonius \cite{Colo16a, Colo16b}. A probability
measure on the space $\Omega$ of control values is fixed. Then an associated
quasi-stationary measure $\eta$ for $Q$ is considered and an entropy notion is
constructed that takes into account information on feedbacks. A significant
relaxation compared to topological invariance entropy is that only invariance
with $\eta$-probability one is required (this was not done in \cite{Colo16a,
Colo16b}). The main results are Theorem \ref{Theorem_conjugacy} showing that
the invariance entropy does not decrease under semi-conjugacy, Theorem
\ref{comparison} showing that the topological invariance entropy is an upper
bound for the metric invariance entropy and Theorems \ref{Theorem_B} and
\ref{Theorem_rel_inv} providing conditions under which the metric invariance
entropy is already determined by certain maximal subsets of approximate
controllability within the interior of $Q$ (i.e., invariant $W$-control sets
with $W:=\mathrm{int}Q$).

A general reference to quasi-stationary measures is the monograph Collett,
Martinez and San Martin \cite{ColMS13}; the survey M\'{e}l\'{e}ard and
Villemonais \cite{MeleV12} presents, in particular, applications to population
dynamics where quasi-stationary distributions correspond to plateaus of
mortality rates. Intuitively speaking, quasi-stationarity measures may exist
when exit from $Q$ occurs with probability one for time tending to infinity,
while in finite time a quasi-stationary behavior develops. A bibliography for
quasi-stationary measures with more than 400 entries is due to Pollett
\cite{Poll15}.

For controllability properties in discrete time, the results by Jakubczyk and
Sontag \cite{JakS90} on reachability are fundamental. Related results, in
particular on control sets, are due to Albertini and Sontag \cite{AlbS93,
AlbS91}, Sontag and Wirth \cite{SonW98} and Wirth \cite{Wirt98NOLCOS},
Patr\~{a}o and San Martin \cite{PatrS07} and Colonius, Homburg and Kliemann
\cite{ColoHK10}. The case of $W$-control sets has only been discussed in the
continuous-time case, cf. Colonius and Lettau \cite{ColoL16}. Although many
properties of control sets and $W$-control sets in discrete time are analogous
to those in continuous time, some additional difficulties occur. In
particular, in the proofs one has to replace the interior of a control set by
its transitivity set or its closely related core.

System (\ref{0.1}) together with the measure $\nu$ on $\Omega$ generates a
random dynamical system. Metric and topological entropy of such systems have
been intensely studied, see, e.g., Bogensch\"{u}tz \cite{Boge92}. If the
control range $\Omega$ in (\ref{0.1}) is a finite set, say $\Omega
=\{1,\ldots,p\}$, then the associated right hand sides $g_{i}=f(\cdot
,i),i\in\Omega$, generate a semigroup of continuous maps acting on $M$. The
metric and topological entropy theory of finitely generated semigroups acting
on compact metric spaces has recently found interest, see, e.g., Rodrigues and
Varandas \cite{RodrV16}.

The contents of this paper is as follows: In Section \ref{Section2}
definitions of metric invariance entropy are presented and discussed. In
particular, the behavior under measurable transformations is characterized and
it is shown that the topological invariance entropy is an upper bound for
metric invariance entropy. Section \ref{Section_Coder} relates the metric
invariance entropy to properties of coder-controllers rendering $Q$ invariant.
Section \ref{Section3} presents conditions ensuring that this entropy is
already determined on a subset $K$ which is invariant in $Q$, i.e., a set
which cannot be left by the system without leaving $Q$. In Section
\ref{Section4}, invariant $W$-control sets are introduced and their properties
are analyzed. The union of their closures yields a set $K$ satisfying the
conditions derived earlier guaranteeing that the metric invariance entropy of
$Q$ coincides with the metric invariance entropy of $K$. Examples
\ref{Example1} and \ref{Example2} illustrate some of the concepts in simple situations.

\textbf{Notation. }Given a probability measure $\mu$ we say that a property
holds for $\mu$-a.a. (almost all) points if it is valid outside a set of $\mu
$-measure zero.

\section{Definition of metric invariance entropy\label{Section2}}

In this section, we present definitions of metric invariance entropy and
discuss their motivation. First we recall entropy of dynamical systems which
also serves to introduce some notation.

Let $\mu$ be a probability measure on a space $X$ endowed with a $\sigma
$-algebra $\mathfrak{F}$. For every finite partition $\mathcal{P}%
=\{P_{1},\ldots,P_{n}\}$ of $X$ into measurable sets the entropy is defined as
$H_{\mu}(\mathcal{P})=-\sum_{i}\phi\left(  \mu(P_{i})\right)  $, where
$\phi(x)=x\log x,x\in(0,1]$ with $\phi(0)=0$. The entropy specifies the
expected information gained from the outcomes in $\mathcal{P}$ of an
experiment, or the amount of uncertainty removed upon learning the
$\mathcal{P}$-address of a randomly chosen point. For a dynamical system
generated by a continuous map $T$ on a compact metric space $X$ one considers
an invariant measure on the Borel $\sigma$-algebra $\mathcal{B}(X)$, i.e.,
$\mu(T^{-1}E)=\mu(E)$ for all $E\in\mathcal{B}(X)$. For a finite partition
$\mathcal{P}$ of $X$ and $j\in\mathbb{N}$ one finds with $T^{-j}%
\mathcal{P}:=\{T^{-j}P\left\vert P\in\mathcal{P}\right.  \}$ that%
\[
\mathcal{P}_{n}:=\bigvee\nolimits_{j=0}^{n-1}T^{-j}\mathcal{P=P}\vee
T^{-1}\mathcal{P}\vee\cdots\vee T^{-(n-1)}\mathcal{P}%
\]
again is a finite partition of $X$ (for two collections $\mathfrak{A}$ and
$\mathfrak{B}$ of sets the join is $\mathfrak{A}\vee\mathfrak{B}=\{A\cap
B\left\vert A\in\mathfrak{A}\text{ and }B\in\mathfrak{B}\right.  \}$). The
entropy of $T$ with respect to the partition $\mathcal{P}$ is $h_{\mu
}(T,\mathcal{P}):=\lim_{n\rightarrow\infty}\frac{1}{n}H_{\mu}\left(
\mathcal{P}_{n}\right)  $. Using conditional entropy, one can also write
\begin{equation}
H_{\mu}\left(  \mathcal{P}_{n}\right)  =\sum_{i=0}^{n-1}H_{\mu}\left(
\mathcal{P}_{i+1}\left\vert \mathcal{P}_{i}\right.  \right)  .
\label{entropy_dyn}%
\end{equation}
The metric entropy of $T$ is $h_{\mu}(T):=\sup_{\mathcal{P}}h_{\mu
}(T,\mathcal{P})$, where the supremum is taken over all finite partitions
$\mathcal{P}$ of $X$, i.e., it is the total information generated by the
dynamical system generated by $T$.

This concept has to be modified when we want to determine the minimal
information that is needed to make a subset $Q$ of the state space of a
control system (\ref{0.1}) invariant under feedbacks. We suppose that a closed
set $Q\subset M$ is given and fix a probability measure $\nu$ on the Borel
$\sigma$-algebra $\mathcal{B}(\Omega)$ of the control range $\Omega$. Let
$p(x,A)=\nu\left\{  \omega\in\Omega\left\vert f(x,\omega)\in A\right.
\right\}  ,x\in M,\,A\subset M$, be the associated Markov transition
probabilities. A quasi-stationary measure with respect to $Q$ of $M$ is a
probability measure $\eta$ on $\mathcal{B}(M)$ such that for some $\rho
\in(0,1]$%
\begin{equation}
\rho\eta(A)=\int_{Q}p(x,A)\eta(dx)\text{ for all }A\in\mathcal{B}(Q).
\label{quasi}%
\end{equation}
The measure $\eta$ is stationary if and only if $\rho=1$. With $A=Q$ one
obtains $\rho=\int_{Q}p(x,Q)\eta(dx)$ and the support $\mathrm{supp}\eta$ is
contained in $Q$. Results on the existence of quasi-stationary measures are
given, e.g., in Collett, Martinez and San Martin \cite[Proposition 2.10 and
Theorem 2.11]{ColMS13} and Colonius \cite[Theorem 2.9]{Colo16a}.

With the shift $\theta:\mathcal{U}\rightarrow\mathcal{U},(u_{k})_{k\geq
0}\mapsto(u_{k+1})_{k\geq0}$, control system (\ref{0.1}) can equivalently be
described by the continuous skew product map%
\begin{equation}
S:\mathcal{U}\times M\rightarrow\mathcal{U}\times M,(u,x)\mapsto(\theta
u,f(x,u_{0})), \label{T_random}%
\end{equation}
where $\mathcal{U}=\Omega^{\mathbb{N}}$ is endowed with the product topology.
Then $S^{k}(u,x)=(\theta^{k}u,\varphi(k,x,u))$. A conditionally invariant
measure $\mu$ for the map $S$ with respect to $Q\subset M$ is a probability
measure on the Borel $\sigma$-algebra of $\mathcal{U}\times M$ such that
$0<\rho:=\mu(S^{-1}(\mathcal{U}\times Q)\cap(\mathcal{U}\times Q))\leq1$ and%
\begin{equation}
\rho\mu(B)=\mu(S^{-1}B\cap(\mathcal{U}\times Q))\text{ for all }%
B\in\mathcal{B}(\mathcal{U}\times M)\text{.} \label{cond_inv}%
\end{equation}
We write $S_{Q}:=S_{\left\vert \mathcal{U}\times Q\right.  }:\mathcal{U}\times
Q\rightarrow\mathcal{U}\times M$ for the restriction. Then the condition in
(\ref{cond_inv}) can be written as $\rho\mu(B)=\mu(S_{Q}^{-1}B)$. For
$k\in\mathbb{N}$ the measure $\mu$ is conditionally invariant for $S_{Q}^{k}$
with constant $\rho^{k}$ and, in particular, $\rho^{-k}\mu$ is a probability
measure on $S_{Q}^{-k}(\mathcal{U}\times Q)$.

If $\eta$ is a quasi-stationary measure, one finds that with the product
measure $\nu^{\mathbb{N}}$ on $\mathcal{U}=\Omega^{\mathbb{N}}$ the measure
$\mu=\nu^{\mathbb{N}}\times\eta$ is a probability measure on the product space
$\mathcal{U}\times M$ satisfying (\ref{cond_inv}), cf. \cite[Proposition
2.8]{Colo16a}. In the present paper only conditionally invariant measures of
this form are considered and $\nu$ will be fixed (cf., e.g., Demers and Young
\cite{DemeY06}, Demers \cite{Demer15} for results on general conditionally
invariant measures). Often we will suppress the dependence on $\nu$ and only
indicate the dependence on the quasi-stationary measure $\eta$.

Next we construct certain partitions for subsets of $\mathcal{U}\times Q$
whose entropy with respect to $\mu=\nu^{\mathbb{N}}\times\eta$ will be used to
define metric invariance entropy.

\begin{definition}
\label{Definition_inv_part}For a closed subset $Q\subset M$ an invariant
$(Q,\eta)$-partition $\mathcal{C}_{\tau}=\mathcal{C}_{\tau}(\mathcal{P},F)$ is
given by $\tau\in\mathbb{N}$, a finite partition $\mathcal{P}$ of $Q$ into
Borel measurable sets and a map $F:\mathcal{P}\rightarrow\Omega^{\tau}$
assigning to each set $P$ in $\mathcal{P}$ a control function such that%
\begin{equation}
\varphi(k,x,F(P))\in Q\text{ for }k\in\{1,\ldots,\tau\}\text{ and }%
\eta\text{-a.a. }x\in P. \label{inv_part}%
\end{equation}

\end{definition}

When no misunderstanding can occur, we just talk about invariant
$Q$-partitions or just invariant partitions. Clearly, condition
(\ref{inv_part}) means that
\[
\eta\{x\in P\left\vert \varphi(k,x,F(P))\in Q\text{ for }k=1,\ldots
,\tau\right.  \}=\eta(P).
\]
Fix an invariant $(Q,\eta)$-partition $\mathcal{C}_{\tau}=\mathcal{C}_{\tau
}(\mathcal{P},F)$ with $\mathcal{P}=\{P_{1},\ldots,P_{q}\}$. Abbreviate
$F_{i}:=F(P_{i})\in\Omega^{\tau},i=1,\ldots,q$, and define for every word
$a:=[a_{0},a_{1},\ldots,a_{n-1}],\allowbreak n\in\mathbb{N}$, with $a_{j}%
\in\{1,\ldots,q\}$ a control function $u_{a}$ on $\{0,\ldots,n\tau-1\}$ by
applying these feedback maps one after the other: for $i=0,\ldots,n-1$ and
$k=0,\ldots,\tau-1$%
\begin{equation}
\left(  u_{a}\right)  _{i\tau+k}:=\left(  F_{a_{i}}\right)  _{k}\text{.}
\label{v_a}%
\end{equation}
We also write $u_{a}:=(F_{a_{0}},F_{a_{1},}\ldots,F_{a_{n-1}})$. A word $a$ is
called $(\eta,\mathcal{C}_{\tau})$-\textit{admissible} if
\begin{equation}
\eta\left\{  x\in Q\left\vert \varphi(i\tau,x,u_{a})\in P_{a_{i}}\text{ for
}i=0,1,\ldots,n-1\right.  \right\}  >0. \label{adm1}%
\end{equation}
Note that for $\eta$-a.a. $x$ it follows that $\varphi(k,x,u_{a})\in Q$ for
$k=0,\ldots,n\tau$ if $\varphi(i\tau,x,u_{a})\in P_{a_{i}}$ for $i=0,\ldots
,n-1$. If $\eta$ and $\mathcal{C}_{\tau}$ are clear from the context, we just
say that $a$ is admissible. The admissible words describe the sequences of
partition elements under\ the feedbacks associated with $\mathcal{C}_{\tau}$
which are followed with positive probability.

For $P\in\mathcal{P}$ we define%
\begin{equation}
A(P,\eta)=\left\{  u\in\mathcal{U}\left\vert \varphi(k,x,u)\in Q\text{ for
}k=1,\ldots,\tau\text{ and }\eta\text{-a.a. }x\in P\right.  \right\}  \times P
\label{A_P}%
\end{equation}
and%
\[
\mathfrak{A}(\mathcal{C}_{\tau},\eta)=\{A(P,\eta)\left\vert P\in
\mathcal{P}\right.  \}\text{ with union }\mathcal{A}(\mathcal{C}_{\tau}%
,\eta)=\bigcup_{P\in\mathcal{P}}A(P,\eta).
\]
Here and in the following the dependence on $\mathcal{C}_{\tau}$ (actually,
these sets only depend on $(\mathcal{P},\tau)$) or $\eta$ is omitted, if it is
clear from the context. The controls $u$ in (\ref{A_P}) can be considered as
constant parts of feedbacks keeping $\eta$-almost all $x\in P$ in $Q$ up to
time $\tau$.

\begin{lemma}
The sets $A(P,\eta)$ defined in (\ref{A_P}) are Borel measurable, hence
$\mathfrak{A}$ is a measurable partition of $\mathcal{A}$ which, in general,
is a proper subset of $\mathcal{U}\times Q$.
\end{lemma}

\begin{proof}
Clearly the sets $A(P,\eta)$ are pairwise disjoint, hence it only remains to
show measurability. We only prove this for the case $\tau=1$, where it
suffices to show that
\[
\left\{  u\in\Omega\left\vert f(x,u)\in Q\text{ for }\eta\text{-a.a. }x\in
P\right.  \right\}  =\left\{  u\in\Omega\left\vert \eta(f(\cdot,u)^{-1}Q\cap
P)\geq\eta(P)\right.  \right\}
\]
is measurable. First we claim that for compact $K\subset P$ and $\delta>0$ the
set%
\begin{equation}
\left\{  u\in\Omega\left\vert \eta(f(\cdot,u)^{-1}Q\cap K)\geq\eta
(P)-\delta\right.  \right\}  \label{Add_1}%
\end{equation}
is closed. In fact, if this set is nonvoid, let $u_{n}\in\Omega,u_{n}%
\rightarrow u\in\Omega$ with
\[
\eta(f(\cdot,u_{n})^{-1}Q\cap K)\geq\eta(P)-\delta.
\]
The sets $B_{m}$ defined by%
\[
B_{m}:=\bigcup_{i=m}^{\infty}f(\cdot,u_{i})^{-1}Q\cap K,B:=\bigcap
_{m=1}^{\infty}B_{m},
\]
are decreasing and $\eta(B)=\lim_{m\rightarrow\infty}\eta(B_{m})\geq
\eta(P)-\delta$. Furthermore, suppose that a subsequence of $y_{i}\in
f(\cdot,u_{i})^{-1}Q\cap K,i\in\mathbb{N}$, converges to $y$. Since
$f^{-1}(Q,\Omega)$ is closed, it follows that $y\in f(\cdot,u)^{-1}Q\cap K$.
This shows that $B\subset f(\cdot,u)^{-1}Q\cap K$ and hence closedness of the
set in (\ref{Add_1}) follows from
\[
\eta(f(\cdot,u)^{-1}Q\cap K)\geq\eta(B)\geq\eta(P)-\delta.
\]
Since $\Omega$ is a compact metric space, regularity of the probability
measure $\eta$ implies that there are compact sets $K_{n}\subset P$ with
$\eta(P\setminus K_{n})\leq\frac{1}{n}$ (cf. Viana and Oliveira
\cite[Proposition A.3.2]{VianO16}). Hence%
\[
\left\{  u\in\Omega\left\vert \eta(f(\cdot,u)^{-1}Q\cap P)\geq\eta(P)\right.
\right\}  \subset\left\{  u\in\Omega\left\vert \eta(f(\cdot,u)^{-1}Q\cap
K_{n})\geq\eta(P)-1/n\right.  \right\}  ,
\]
and it follows that
\[
\left\{  u\in\Omega\left\vert \eta(f(\cdot,u)^{-1}Q\cap P)\geq\eta(P)\right.
\right\}  =\bigcap_{n=1}^{\infty}\left\{  u\in\Omega\left\vert \eta
(f(\cdot,u)^{-1}Q\cap K_{n})\geq\eta(P)-\frac{1}{n}\right.  \right\}  .
\]
Thus the set on the left hand side is measurable as countable intersection of
closed sets.
\end{proof}

A sequence $(A_{0},\ldots,A_{n-1})$ of sets in $\mathfrak{A}$ is called
$\mathcal{C}_{\tau}$-admissible (or a $\mathcal{C}_{\tau}$-itinerary), if
there is an admissible word $a=[a_{0},\ldots,a_{n-1}]$ of length $n$ with
$A_{i}=A(P_{a_{i}})\in\mathfrak{A}$ for all $i$. Then also the set%
\begin{equation}
D_{a}=A_{0}\cap S^{-\tau}A_{1}\cap\cdots\cap S^{-(n-1)\tau}A_{n-1}\in
\bigvee_{i=0}^{n-1}S_{Q}^{-i\tau}\mathfrak{A} \label{D_a}%
\end{equation}
is called admissible. Only the sets $D_{a}$ with $\mu(D_{a})>0$ will be
relevant (as usual, if $\mu(D_{a})=0$, this set is simply omitted in the following).

Note that $\mu$-a.a. $(u,x)\in D_{a}$ satisfy $\varphi(k,x,u)\in Q$ for
$k=0,\ldots,n\tau$. The collection of all sets $D_{a}$ is%
\begin{equation}
\mathfrak{A}_{n}:=\left\{  D_{a}\in\bigvee_{i=0}^{n-1}S^{-i\tau}%
\mathfrak{A}\left\vert a\text{ admissible}\right.  \right\}  ,~\mathcal{A}%
_{n}:=\bigcup_{a\text{ admissible}}D_{a}\subset S_{Q}^{-(n-1)\tau}%
(\mathcal{U}\times Q). \label{A_n}%
\end{equation}
Observe that $\mathfrak{A}_{1}=\mathfrak{A}$ and that $\mathfrak{A}_{n}$ is a
measurable partition of $\mathcal{A}_{n}$ and, for convenience, we set
$\mathfrak{A}_{0}=\mathcal{U}\times Q$. Note that the inclusion $\mathcal{A}%
_{n+1}\subset\mathcal{A}_{n}$ holds for all $n\in\mathbb{N}$, and that, in
general, it is proper.

\begin{remark}
For invariant $(Q,\eta)$-partitions the inclusion%
\begin{equation}
\mathfrak{A}_{n+m}\subset\mathfrak{A}_{n}\vee S^{-n\tau}\mathfrak{A}%
_{m},n,m\in\mathbb{N}, \label{incl}%
\end{equation}
does not hold, in general (in contrast to Colonius \cite[Lemma 3.3]{Colo16a}
where only existence of a trajectory following a sequence of partition
elements is required). The problem is that for an $(\eta,\mathcal{C}_{\tau}%
)$-admissible word $a=[a_{0},\ldots,a_{n-1},a_{n},\ldots,a_{n+m-1}]$ the word
$[a_{n},\ldots,a_{n+m-1}]$ need not be $(\eta,\mathcal{C}_{\tau})$-admissible:
Certainly, the inequality%
\begin{align*}
&  \eta\left\{  y\in Q\left\vert \varphi(i\tau,y,u_{[a_{n},,\ldots,n+m-1]})\in
P_{a_{n+i}}\text{ for }i=0,\ldots,m-1\right.  \right\} \\
&  \geq\eta\left\{  \varphi(n\tau,x,u_{a})\in Q\left\vert x\in Q,\varphi
(i\tau,x,u_{a})\in P_{a_{i}}\text{ for }i=0,\ldots,n+m-1\right.  \right\}
\end{align*}
holds, but the term on the right hand side need not be positive. This could be
guaranteed by changing the definition of $(\eta,\mathcal{C}_{\tau}%
)$-admissible words $a=[a_{0},\ldots,a_{n-1}]$ to: For all $j=0,1,\ldots,n-1$%
\[
\eta\left\{  y\in Q\left\vert \varphi(i\tau,y,u_{a})\in P_{a_{i}}\text{ for
}i=j,j+1,\ldots,n-1\right.  \right\}  >0.
\]
We do not adapt this definition since the inclusion (\ref{incl}) is not needed below.
\end{remark}

The direct way to define a notion of metric invariance entropy is to consider
the entropy of the partitions $\mathfrak{A}_{n}$ of the sets $\mathcal{A}_{n}%
$. An alternative is to consider the additional information in every step. We
start with the first choice.

We consider the entropy $H_{\rho^{-(n-1)\tau}\mu}(\mathfrak{A}_{n}%
(\mathcal{C}_{\tau}))$ of $\mathfrak{A}_{n}(\mathcal{C}_{\tau})$ in
$\mathcal{A}_{n}(\mathcal{C}_{\tau})\subset S_{Q}^{-(n-1)\tau}(\mathcal{U}%
\times Q)$ with respect to the probability measure $\rho^{-(n-1)\tau}\mu$ and
then take the average of the required information as time tends to $\infty$ to
get the invariance $\mu$-entropy of $\mathcal{C}_{\tau}$,%
\[
h_{\mu}(\mathcal{C}_{\tau},Q):=\underset{n\rightarrow\infty}{\lim\sup}\frac
{1}{n\tau}H_{\rho^{-(n-1)\tau}\mu}(\mathfrak{A}_{n}(\mathcal{C}_{\tau})).
\]

\begin{definition}
\label{Definition_main}Let $\eta$ be a quasi-stationary measure on a closed
set $Q$ for a measure $\nu$ on $\Omega$ and set $\mu=\nu^{\mathbb{N}}%
\times\eta$. The invariance entropy for control system (\ref{0.1}) is%
\begin{equation}
h_{\mu}(Q):=\underset{\tau\rightarrow\infty}{\lim\sup}\inf_{\mathcal{C}_{\tau
}}h_{\mu}(\mathcal{C}_{\tau},Q), \label{Main_entropy}%
\end{equation}
where for fixed $\tau\in\mathbb{N}$ the infimum is taken over all invariant
$(Q,\eta)$-partitions $\mathcal{C}_{\tau}=\mathcal{C}_{\tau}(\mathcal{P},F)$.
If no invariant $(Q,\eta)$-partition $\mathcal{C}_{\tau}$ exists, we set
$h_{\mu}(Q):=\infty$.
\end{definition}

The following remarks comment on this definition.

\begin{remark}
\label{Remark_partition}An objection to the consideration of $H_{\rho
^{-(n-1)\tau}\mu}(\mathfrak{A}_{n})$ might be that $\mathfrak{A}_{n}$ is not a
partition of $S_{Q}^{-(n-1)\tau}(\mathcal{U}\times Q)$, while $\rho
^{-(n-1)\tau}\mu$ is a probability measure on this space. However, one may add
to the collection $\mathfrak{A}_{n}(\mathcal{C}_{\tau})$ the complement
\[
Z_{n}:=\left(  S_{Q}^{-(n-1)\tau}(\mathcal{U}\times Q)\right)  \setminus
\mathcal{A}_{n}(\mathcal{C}_{\tau}).
\]
Thus one obtains a partition $\mathfrak{A}_{n}\cup\{Z_{n}\}$ of $S_{Q}%
^{-(n-1)\tau}(\mathcal{U}\times Q)$ with entropy%
\[
H_{\rho^{-(n-1)\tau}\mu}\left(  \mathfrak{A}_{n}(\mathcal{C}_{\tau}%
)\cup\{Z_{n}\}\right)  =H_{\rho^{-(n-1)\tau}\mu}(\mathfrak{A}_{n}%
(\mathcal{C}_{\tau}))-\phi\left(  \rho^{-(n-1)\tau}\mu(Z_{n})\right)  .
\]
For each $n\in\mathbb{N}$ the second summand is bounded by $1/e=\max
_{x\in\lbrack0,1]}(-\phi(x))$ and hence%
\[
h_{\mu}(\mathcal{C}_{\tau})=\underset{n\rightarrow\infty}{\lim\sup}\frac
{1}{n\tau}H_{\rho^{-(n-1)\tau}\mu}(\mathfrak{A}_{n}(\mathcal{C}_{\tau
}))=\underset{n\rightarrow\infty}{\lim\sup}\frac{1}{n\tau}H_{\rho^{-(n-1)\tau
}\mu}\left(  \mathfrak{A}_{n}(\mathcal{C}_{\tau})\cup\{Z_{n}\}\right)  .
\]
This shows that $h_{\mu}(\mathcal{C}_{\tau})$ is given by the entropy of bona
fide partitions.
\end{remark}

\begin{remark}
Definition (\ref{Main_entropy}) ensures that $\tau\rightarrow\infty$. This
will be needed in the proof of Theorem \ref{TheoremA_alt}(ii). Instead of the
limit superior for $n\rightarrow\infty$ and $\tau\rightarrow\infty$ one also
might consider the limit inferior. However, the limit superior is advantageous
in Theorem \ref{TheoremA_alt}(iii). For topological invariance entropy, one
takes instead an infimum over all invariant open covers (where the partition
$\mathcal{P}$ is replaced by an open cover of $Q$). Then it follows that it
suffices to take the limit for $\tau\rightarrow\infty$, cf. Kawan
\cite[Theorem 2.3 and its proof]{Kawa13}.
\end{remark}

\begin{remark}
\label{Remark_null}If the sets in invariant $(Q,\eta)$-partitions
$\mathcal{C}_{\tau}$ and $\mathcal{C}_{\tau}^{\prime}$ coincide modulo $\eta
$-null sets, the entropies $H_{\rho^{-(n-1)\tau}\mu}(\mathfrak{A}%
_{n}(\mathcal{C}_{\tau}))$ and $H_{\rho^{-(n-1)\tau}\mu}(\mathfrak{A}%
_{n}(\mathcal{C}_{\tau}^{\prime})),n\in\mathbb{N}$, coincide. Hence it
suffices to specify a partition of $Q$ outside of a set of $\eta$-measure zero.
\end{remark}

\begin{remark}
For a stationary measure $\eta$ the trivial partition of $Q$ yields an
invariant partition. In fact, for every $\tau>0$ there is an invariant
$(Q,\eta)$-partition $\mathcal{C}_{\tau}=\mathcal{C}_{\tau}(\{Q\},F)$, where
the control $F(Q)\in\mathcal{U}$ can be chosen arbitrarily in a set of full
$\nu^{\mathbb{N}}$-measure in $\mathcal{U}$. Thus the associated metric
invariance entropy vanishes. This is seen as follows: Assume, contrary to the
assertion, that there is a set $\ \mathcal{U}_{0}\subset\mathcal{U}$ with
$\nu^{\mathbb{N}}(\mathcal{U}_{0})>0$ such that for every $u\in\mathcal{U}%
_{0}$%
\begin{equation}
\eta\{x\in Q\left\vert \varphi(k,x,u)\in Q\text{ for }k=1,\ldots,\tau\right.
\}<\eta(Q)=1. \label{a1}%
\end{equation}
We may assume that there is $k\in\{1,\ldots,\tau\}$ such that for every
$u\in\mathcal{U}_{0}$%
\[
\eta\{x\in Q\left\vert \varphi(k,x,u)\in Q\right.  \}<1.
\]
By invariance of $\nu^{\mathbb{N}}\times\eta$ Fubini's theorem yields the
contradiction%
\begin{align*}
1  &  =\int_{\mathcal{U}\times Q}S_{Q}^{-k}(\mathcal{U}\times Q)(\nu
^{\mathbb{N}}\times\eta)(du,dx)=\int_{\mathcal{U}}\int_{Q}\chi
_{\{(u,x)\left\vert \varphi(k,x,u)\in Q\right.  \}}\eta(dx)\nu^{\mathbb{N}%
}(du)\\
&  <\int_{\mathcal{U}_{0}}\eta(Q)\nu^{\mathbb{N}}(du)+\int_{\mathcal{U}%
\setminus\mathcal{U}_{0}}\eta(Q)\nu^{\mathbb{N}}(du)=\int_{\mathcal{U}}%
\eta(Q)\nu^{\mathbb{N}}(du)=1.
\end{align*}

\end{remark}

\begin{remark}
If there exists an invariant partition, then invariant partitions with
arbitrarily large time step $\tau$ exist. It suffices to see that for every
invariant partition $\mathcal{C}_{\tau}=\mathcal{C}_{\tau}(\mathcal{P},F)$
there exists an invariant partition $\mathcal{C}_{2\tau}=\mathcal{C}_{2\tau
}(\mathcal{P}^{2},F^{2})$. In fact, for $P_{i},P_{j}\in\mathcal{P}$ let%
\[
P_{ij}:=\{x\in P_{i}\left\vert \varphi(\tau,x,F(P_{i}))\in P_{j}\right.
\}=P_{i}\cap\varphi(\tau,\cdot,F(P_{i}))^{-1}P_{j}\}.
\]
This yields a partition of $Q$ given by $\mathcal{P}^{2}:=\{P_{ij}\left\vert
P_{i},P_{j}\in\mathcal{P}\right.  \}$. Define feedbacks $F^{2}:\mathcal{P}%
^{2}\rightarrow\Omega^{2\tau}$ by%
\[
F^{2}(P_{ij})(r):=\left\{
\begin{array}
[c]{ccc}%
F(P_{i})(r) & \text{for} & r=0,\ldots,\tau-1\\
F(P_{j})(r-\tau) & \text{for} & \text{ }r=\tau,\ldots,2\tau-1
\end{array}
\right.  .
\]
Then $\mathcal{C}_{2\tau}=(\mathcal{P}^{2},F^{2})$ is an invariant partition.
\end{remark}

An alternative concept of metric invariance entropy can be based on the
additional information gained in every time step (this was proposed in
Colonius \cite{Colo16a} and is slightly reformulated below). The following
construction has to take into account that the space $\mathcal{A}_{n}$ that is
partitioned decreases in every time step.

Let an invariant partition $\mathcal{C}_{\tau}$ be given. Then the partition
$\mathfrak{A}_{n}=\mathfrak{A}_{n}(\mathcal{C}_{\tau})$ of $\mathcal{A}_{n}$
induces a partition of $\mathcal{A}_{n+1}$: For $D\in\mathfrak{A}_{n}$ let%
\begin{equation}
\mathfrak{A}_{n+1}(D)=\{E\in\mathfrak{A}_{n+1}\left\vert E\cap D\not =%
\emptyset\right.  \},\mathcal{A}_{n+1}(D)=\bigcup\nolimits_{E\in
\mathfrak{A}_{n+1}(D)}E. \label{A(D)}%
\end{equation}
Since $E\cap D\not =\emptyset$ implies $E\subset D$ and the sets
$\mathcal{A}_{n+1}(D)$ are mutually disjoint, one obtains an induced partition%
\begin{equation}
\mathfrak{A}_{n}^{n+1}:=\left\{  \mathcal{A}_{n+1}(D)\left\vert D\in
\mathfrak{A}_{n}\right.  \right\}  \text{ of }\mathcal{A}_{n+1}=\bigcup
\nolimits_{D\in\mathfrak{A}_{n}}\mathcal{A}_{n+1}(D). \label{induced}%
\end{equation}
Clearly, $\mathfrak{A}_{n+1}$ is a refinement of $\mathfrak{A}_{n}^{n+1}$. The
information from $\mathfrak{A}_{n}$ that is relevant for $\mathfrak{A}_{n+1}$
comes from the partition $\mathfrak{A}_{n}^{n+1}$. Assuming that the
information encoded in $\mathfrak{A}_{n}^{n+1}$ is known at the time step $n$,
the incremental information is the conditional entropy of $\mathfrak{A}_{n+1}$
given $\mathfrak{A}_{n}^{n+1}$ (with respect to $\rho^{-n\tau}\mu$). For every
$n\in\mathbb{N}$%
\[
H_{\rho^{-n\tau}\mu}(\mathfrak{A}_{n+1})=H_{\rho^{-n\tau}\mu}(\mathfrak{A}%
_{n}^{n+1})+H_{\rho^{-n\tau}\mu}(\mathfrak{A}_{n+1}\left\vert \mathfrak{A}%
_{n}^{n+1}\right.  ),
\]
where the conditional entropy of $\mathfrak{A}_{n+1}$ given $\mathfrak{A}%
_{n}^{n+1}$ is%
\begin{equation}
H_{\rho^{-n\tau}\mu}(\mathfrak{A}_{n+1}\left\vert \mathfrak{A}_{n}%
^{n+1}\right.  )=-\sum_{D\in\mathfrak{A}_{n}}\rho^{-n\tau}\mu(\mathcal{A}%
_{n+1}(D))\sum_{E\in\mathfrak{A}_{n+1}}\phi\left(  \frac{\mu(D\cap E)}%
{\mu(\mathcal{A}_{n+1}(D))}\right)  . \label{condition_fb0}%
\end{equation}
(Observe that in the argument of $\phi$ one may multiply numerator and
denominator by $\rho^{-n\tau}$.) Taking the average incremental information
one arrives at the following notion.

\begin{definition}
\label{Def_cond}Let $\eta$ be a quasi-stationary measure on a closed set $Q$
for a measure $\nu$ on $\Omega$ and set $\mu=\nu^{\mathbb{N}}\times\eta$. For
an invariant $(Q,\eta)$-partition $\mathcal{C}_{\tau}=\mathcal{C}_{\tau
}(\mathcal{P},F)$ define the incremental invariance entropy of $\mathcal{C}%
_{\tau}$ by%
\[
h_{\mu}^{inc}(\mathcal{C}_{\tau},Q)=\underset{n\rightarrow\infty}{\lim\sup
}\frac{1}{n\tau}\sum_{j=0}^{n-1}H_{\rho^{-j\tau}\mu}\left(  \mathfrak{A}%
_{j+1}(\mathcal{C}_{\tau})\left\vert \mathfrak{A}_{j}^{j+1}(\mathcal{C}_{\tau
})\right.  \right)  ,
\]
and define the incremental invariance entropy for control system (\ref{0.1})
by%
\begin{equation}
h_{\mu}^{inc}(Q):=\underset{\tau\rightarrow\infty}{\lim\sup}\inf
_{\mathcal{C}_{\tau}}h_{\mu}^{inc}(\mathcal{C}_{\tau},Q),
\label{Main_entropy_inc}%
\end{equation}
where the infimum is taken over all invariant $(Q,\eta)$-partitions
$\mathcal{C}_{\tau}(\mathcal{P},F)$. If no invariant $(Q,\eta)$-partition
$\mathcal{C}_{\tau}$ exists, we set $h_{\mu}^{inc}(Q):=\infty$.
\end{definition}

Regrettably, a formula analogous to (\ref{entropy_dyn}) for dynamical systems
is not available for invariance entropy of control systems. Hence the relation
between the invariance entropy and the incremental invariance entropy remains
unknown. The following proposition only describes a relation between the
entropy of $\mathfrak{A}_{n}$ and of the induced partition $\mathfrak{A}%
_{n}^{n+1}$.

\begin{proposition}
There is $K\in\mathbb{N}$ such that for all $n\in\mathbb{N}$%
\[
H_{\rho^{-n\tau}\mu}(\mathfrak{A}_{n}^{n+1})\leq H_{\rho^{-n\tau}%
}(\mathfrak{A}_{n})+K/e.
\]

\end{proposition}

\begin{proof}
There is $K\in\mathbb{N}$ such that for every probability measure $m$ there
are at most $K$ mutually disjoint sets $A_{1},\ldots,A_{K}$ with
$m(A_{i})>\rho^{\tau}/e$, since $\sum_{i=1}^{K}m(A_{i})\leq1$. In particular,
for every $n\in\mathbb{N}$ there are at most $K$ sets $D\in\mathfrak{A}_{n}$
such that $\rho^{-(n-1)\tau}\mu(D)>\rho^{\tau}/e$, since every $D\in
\mathfrak{A}_{n}$ is contained in $S_{Q}^{-(n-1)\tau}(\mathcal{U}\times Q)$
and $\rho^{-(n-1)\tau}\mu$ is a probability measure on this set. This
inequality is equivalent to $\rho^{-n\tau}\mu(D)>1/e$. Let $\mathfrak{A}%
_{n}^{big}$ be the set of elements in $\mathfrak{A}_{n}$ with $\rho^{-n\tau
}\mu(D)>1/e$. The other elements $D$ in $\mathfrak{A}_{n}$ satisfy%
\[
\rho^{-n\tau}\mu(\mathcal{A}_{n+1}(D))\leq\rho^{-n\tau}\mu(D)\leq1/e.
\]
Since $\phi$ is monotonically decreasing on $[0,1/e]$ it follows that
$\phi(\rho^{-n\tau}\mu(\mathcal{A}_{n+1}(D)))\geq\phi(\rho^{-n\tau}\mu(D)) $
and hence%
\[
H_{\rho^{-n\tau}\mu}(\mathfrak{A}_{n}^{n+1})\leq H_{\rho^{-n\tau}%
}(\mathfrak{A}_{n})-K\min\phi\leq H_{\rho^{-n\tau}}(\mathfrak{A}_{n})+K/e.
\]

\end{proof}

Next we analyze the behavior of both notions of invariance entropy under
measure preserving transformations (cf. Walters \cite[\S 2.3]{Walt82}). For
notational simplicity, we suppose that the control ranges and the measures on
them coincide.

\begin{definition}
Consider two control systems of the form (\ref{0.1}) on $M_{1}$ and $M_{2}$,
respectively, given by%
\begin{equation}
x_{k+1}=f_{1}(x_{k},u_{k})\text{ and }y_{k+1}=f_{2}(y_{k},u_{k})\text{ with
}(u_{k})\in\mathcal{U}=\Omega^{\mathbb{N}}. \label{factor}%
\end{equation}
Let $\nu$ be a probability measure on $\Omega$ and suppose that $\eta_{1}$ and
$\eta_{2}$ are corresponding quasi-stationary measures with respect to closed
subsets $Q_{1}\subset M_{1}$ and $Q_{2}\subset M_{2}$, respectively. We say
that $(f_{2},\eta_{2})$ is semi-conjugate to $(f_{1},\eta_{1})$, if there are
subsets $\hat{\Omega}\subset\Omega$ and $\hat{Q}_{i}\subset Q_{i} $ of full
$\nu$-measure and full $\eta_{i}$-measure, $i=1,2$, resp., with the following properties:

(i) there exists a measurable map $\pi:\hat{Q}_{1}\rightarrow\hat{Q}_{2}$ such
that $\pi$ maps $\eta_{1}$ onto $\eta_{2}$, i.e.,%
\begin{equation}
(\pi_{\ast}\eta_{1})(B):=\eta_{1}(\pi^{-1}B)=\eta_{2}(B)\text{ for all
}B\subset\hat{Q}_{2}, \label{conjugacy2}%
\end{equation}

(ii) one has $f_{i}(x,\omega)\in\hat{Q}_{i}$ for all $(\omega,x)\in\hat
{\Omega}\times\hat{Q}_{i},i=1,2$, and%
\begin{equation}
\pi(f_{1}(x,\omega))=f_{2}(\pi x,\omega)\text{ for }\omega\in\hat{\Omega
}\text{ and }x\in\hat{Q}_{1}. \label{conjugacy0}%
\end{equation}

\end{definition}

The map $\pi$ is called a semi-conjugacy from $(f_{1},\eta_{1})$ to
$(f_{2},\eta_{2})$. In terms of the solutions, condition (\ref{conjugacy0})
implies that for $\nu^{\mathbb{N}}$-a.a. $u\in\Omega^{\mathbb{N}}$ and
$\eta_{1}$-a.a. $x\in Q_{1}$%
\[
\pi\varphi_{1}(k,x_{0},u)=\varphi_{2}(k,\pi x_{0},u)\text{ for all }%
k\in\mathbb{N}.
\]
With the associated skew product maps $S_{i}(u,x)=(\theta u,f_{i}(x,u_{0}))$
one obtains $S_{i}(\hat{\Omega}^{\mathbb{N}}\times\hat{Q}_{i})\subset
\hat{\Omega}^{\mathbb{N}}\times\hat{Q}_{i},i=1,2$, and for all $(u,x)\in
\hat{\Omega}^{\mathbb{N}}\times\hat{Q}_{1}$%
\begin{align}
\left(  \mathrm{id}_{\mathcal{U}}\times\pi\right)  \circ S_{1}(u,x)  &
=\left(  \mathrm{id}_{\mathcal{U}}\times\pi\right)  (\theta u,f_{1}%
(x,u_{0}))=(\theta u,f_{2}(\pi x,u_{0}))\label{conjugacy3+}\\
&  =S_{2}\circ\left(  \mathrm{id}_{\mathcal{U}}\times\pi\right)
(u,x).\nonumber
\end{align}
If the map $\pi$ is a bimeasurable bijection, we obtain an equivalence
relation called conjugacy. A consequence of the following theorem is that the
metric invariance entropies are invariant under conjugacies.

\begin{theorem}
\label{Theorem_conjugacy}Suppose that for two control systems given by
(\ref{factor}) there is a semi-conjugacy $\pi$ from $(f_{1},\eta_{1})$ to
$(f_{2},\eta_{2})$. Then the constants $\rho_{i}$ coincide and with $\mu
_{i}=\nu^{\mathbb{N}}\times\eta_{i},i=1,2$, the metric invariance entropy
satisfies%
\[
h_{\mu_{1}}(Q_{1})\leq h_{\mu_{2}}(Q_{2})\text{ and }h_{\mu_{1}}^{inc}%
(Q_{1})\leq h_{\mu_{2}}^{inc}(Q_{2}).
\]

\end{theorem}

\begin{proof}
First observe that $\eta_{2}(\hat{Q}_{2})=\eta_{1}(\pi^{-1}\hat{Q}_{2}%
)=\eta_{1}(\hat{Q}_{1})$ and
\[
\mu_{2}=\nu^{\mathbb{N}}\times\eta_{2}=\left(  \mathrm{id}_{\mathcal{U}}%
\times\pi\right)  _{\ast}(\nu^{\mathbb{N}}\times\eta_{1})=\left(
\mathrm{id}_{\mathcal{U}}\times\pi\right)  _{\ast}\mu_{1}.
\]
Furthermore $\rho_{1}=\rho_{2}$, since properties (\ref{conjugacy2}) and
(\ref{conjugacy3+}) imply%
\begin{align*}
\rho_{2}  &  =\rho\mu_{2}(\mathcal{U}\times\hat{Q}_{2})=\mu_{2}\{(u,y)\in
\mathcal{U}\times Q_{2}\left\vert S_{2}(u,y)\in\mathcal{U}\times\hat{Q}%
_{2}\right.  \}\\
&  =\mu_{2}\{(u,y)\in\hat{\Omega}^{\mathbb{N}}\times\hat{Q}_{2}\left\vert
S_{2}(u,y)\in\mathcal{U}\times\hat{Q}_{2}\right.  \}\\
&  =\mu_{1}\left(  \mathrm{id}_{\mathcal{U}}\times\pi\right)  ^{-1}%
\{(u,y)\in\hat{\Omega}^{\mathbb{N}}\times\hat{Q}_{2}\left\vert S_{2}%
(u,y)\in\mathcal{U}\times\hat{Q}_{2}\right.  \}\\
&  =\mu_{1}\{(u,x)\in\hat{\Omega}^{\mathbb{N}}\times\hat{Q}_{1}\left\vert
S_{2}\circ(\mathrm{id}_{\mathcal{U}}\times\pi)(u,x)\in\mathcal{U}\times\hat
{Q}_{2}\right.  \}\\
&  =\mu_{1}\{(u,x)\in\hat{\Omega}^{\mathbb{N}}\times\hat{Q}_{1}\left\vert
(\mathrm{id}_{\mathcal{U}}\times\pi)\circ S_{1}(u,x)\in\mathcal{U}\times
\hat{Q}_{2}\right.  \}\\
&  =\mu_{1}\{(u,x)\in\hat{\Omega}^{\mathbb{N}}\times\hat{Q}_{1}\left\vert
S_{1}(u,x)\in\mathcal{U}\times\hat{Q}_{1}\right.  \}\\
&  =\mu_{1}\{(u,x)\in\mathcal{U}\times Q_{1}\left\vert S_{1}(u,x)\in
\mathcal{U}\times\hat{Q}_{1}\right.  \}\\
&  =\rho_{1}\mu_{1}(\mathcal{U}\times\hat{Q}_{1})=\rho_{1}.
\end{align*}
Let $\mathcal{C}_{2,\tau}=\mathcal{C}_{2,\tau}(\mathcal{P}_{2},F)$ be an
invariant $(Q_{2},\eta_{2})$-partition. Then it follows that $\pi
^{-1}\mathcal{P}_{2}\allowbreak=\{\pi^{-1}P\left\vert P\in\mathcal{P}%
_{2}\right.  \}$ is a measurable partition of $Q_{1}=\pi^{-1}Q_{2}$ modulo
$\eta_{1}$-null sets and we may assume that $\pi^{-1}P\subset\hat{Q}_{1}$ for
all $P$. For $P\in\mathcal{P}_{2}$ it follows that for $\eta_{2}$-a.a. $x\in
P$ one has $\varphi_{2}(k,x,F(P))\in Q_{2}$ for all $k\in\{1,\ldots,\tau\}$ if
and only if for $\eta_{1}$-a.a. $y\in\pi^{-1}P$ one has $y=\pi^{-1}x$\ for
some $x\in P$ and%
\[
\varphi_{1}(k,y,F(P))\in\pi^{-1}\varphi_{2}(k,x,F(P))\in\pi^{-1}Q_{2}%
=Q_{1}\text{ for all }k\in\{1,\ldots,\tau\}
\]
(note that the preimage under $\pi$ of an $\eta_{2}$-null set is an $\eta_{1}%
$-null set). By Remark \ref{Remark_null} it follows that $\mathcal{C}_{1,\tau
}=\mathcal{C}_{1,\tau}(\pi^{-1}\mathcal{P}_{2},F)$ with $F(\pi^{-1}%
P):=F(P),\pi^{-1}P\in\pi^{-1}\mathcal{P}_{1}$, is an invariant $(Q_{1}%
,\eta_{1})$-partition. Then the preimage of the collection $\mathfrak{A}%
(\mathcal{C}_{2,\tau})$ of $\mathcal{U}\times Q_{1}$ equals the collection%
\[
\mathfrak{A}(\mathcal{C}_{1,\tau})=\{\left(  \mathrm{id}_{\mathcal{U}}%
\times\pi\right)  ^{-1}A\left\vert A\in\mathfrak{A}(\mathcal{C}_{2,\tau
})\right.  \}.
\]
Let $a$ be a $(Q_{2},\eta_{2})$-admissible word. Then%
\begin{align*}
0  &  <\eta_{2}\left\{  y\in Q_{2}\left\vert \varphi_{2}(i\tau,y,u_{a})\in
P_{a_{i}}\text{ for }i=0,1,\ldots,n-1\right.  \right\} \\
&  =\eta_{2}\left\{  y\in\hat{Q}_{2}\left\vert \varphi_{2}(i\tau,y,u_{a})\in
P_{a_{i}}\text{ for }i=0,1,\ldots,n-1\right.  \right\} \\
&  =\eta_{1}\left\{  \pi^{-1}y\in\hat{Q}_{1}\left\vert \varphi_{2}%
(i\tau,y,u_{a})\in P_{a_{i}}\text{ for }i=0,1,\ldots,n-1\right.  \right\} \\
&  =\eta_{1}\left\{  x\in\hat{Q}_{1}\left\vert \varphi_{2}(i\tau,\pi
x,u_{a})\in P_{a_{i}}\text{ for }i=0,1,\ldots,n-1\right.  \right\} \\
&  =\eta_{1}\left\{  x\in\hat{Q}_{1}\left\vert \pi\varphi_{1}(i\tau
,x,u_{a})\in P_{a_{i}}\text{ for }i=0,1,\ldots,n-1\right.  \right\} \\
&  =\eta_{1}\left\{  x\in\hat{Q}_{1}\left\vert \varphi_{1}(i\tau,x,u_{a}%
)\in\pi^{-1}P_{a_{i}}\text{ for }i=0,1,\ldots,n-1\right.  \right\}  .
\end{align*}
It follows that $a$ is also $(Q_{1},\eta_{1})$-admissible. The same arguments
show that every $(Q_{1},\eta_{1})$-admissible word is also $(Q_{2},\eta_{2}%
)$-admissible and hence for all $n\in\mathbb{N}$%
\[
\left(  \mathrm{id}_{\mathcal{U}}\times\pi\right)  ^{-1}\mathfrak{A}%
_{n}(\mathcal{C}_{2,\tau})=\mathfrak{A}_{n}(\mathcal{C}_{1,\tau}).
\]
One finds for the entropy%
\[
H_{\rho_{1}^{-(n-1)\tau}\mu_{1}}(\mathfrak{A}_{n}(\mathcal{C}_{1,\tau
}))=H_{\rho_{2}^{-(n-1)\tau}\mu_{2}}(\mathfrak{A}_{n}(\mathcal{C}_{2,\tau
})),n\in\mathbb{N},
\]
and hence $h_{\mu_{1}}(\mathcal{C}_{1,\tau})=h_{\mu_{2}}(\mathcal{C}_{2,\tau
})$.

Taking first the infimum over all invariant $(Q_{2},\eta_{2})$-partitions and
then over all invariant $(Q_{1},\eta_{1})$-partitions one finds that
$h_{\mu_{1}}(Q_{1})\allowbreak\leq h_{\mu_{2}}(Q_{2})$. These arguments also
show that $h_{\mu_{2}}(Q_{1})=\infty$ if $h_{\mu_{1}}(Q_{2})\allowbreak
=\infty$.

For the incremental invariance entropy one similarly finds that for all
$n\in\mathbb{N}$%
\[
\left(  \mathrm{id}_{\mathcal{U}}\times\pi\right)  \mathfrak{A}_{n}%
^{n+1}(\mathcal{C}_{1,\tau})=\mathfrak{A}_{n}^{n+1}(\mathcal{C}_{2,\tau}).
\]
Then it follows that
\[
H_{\rho^{-n\tau}\mu_{1}}(\mathfrak{A}_{n+1}(\mathcal{C}_{1,\tau})\left\vert
\mathfrak{A}_{n}^{n+1}(\mathcal{C}_{1,\tau})\right.  )=H_{\rho^{-n\tau}\mu
_{2}}(\mathfrak{A}_{n+1}(\mathcal{C}_{2,\tau})\left\vert \mathfrak{A}%
_{n}^{n+1}(\mathcal{C}_{2,\tau})\right.  )
\]
and the inequality of the incremental invariance entropies is a consequence.
\end{proof}

\begin{remark}
Observe that the inequalities for invariance entropies under semi-conjugacy
are opposite to the inequalities for entropy of dynamical systems, cf. Viana
and Oliveira \cite[Exercise 9.1.5]{VianO16}. This is due to the fact that we
construct invariant $(Q_{1},\eta_{1})$-partitions from invariant $(Q_{2}%
,\eta_{2})$-partitions and then take the infimum (instead of the supremum) of
partitions. Note also that for topological invariance entropy Kawan
\cite[Proposition 2.13]{Kawa13} constructs from spanning sets of controls for
$Q_{1}$ spanning sets for $Q_{2}$. Then letting the time tend to infinity and
taking the infimum over spanning sets one gets that the invariance entropy of
$Q_{1}$ is greater than or equal to the invariance entropy of $Q_{2}%
$.\smallskip
\end{remark}

To conclude this section we show that the metric invariance entropy is bounded
above by the topological invariance entropy. As in Kawan \cite[Definition 2.2
and Proposition 2.3(ii)]{Kawa13} consider for system (\ref{0.1}) a compact
controlled invariant set $Q\subset M$, i.e., for every $x\in Q$ there is
$\omega_{x}\in\Omega$ with $f(x,\omega_{x})\in Q$. For $\tau\in\mathbb{N}$ a
set $\mathcal{R}\subset\mathcal{U}$ is called $(\tau,Q)$-spanning if for all
$x\in Q$ there is $u\in\mathcal{R}$ with $\varphi(n,x,u)\in Q$ for all
$n=1,\ldots,\tau$. Denote by $r_{inv}(\tau,Q)$ the minimal number of elements
such a set can have (if no finite $(\tau,Q)$-spanning set exists,
$r_{inv}(\tau,Q):=\infty$). The topological invariance entropy is defined by%
\[
h_{inv}(Q):=\lim_{\tau\rightarrow\infty}\frac{1}{\tau}\log r_{inv}(\tau,Q).
\]
In order to relate this notion to metric invariance entropy, we use a
characterization of topological invariance entropy by invariant partitions
(the original definition due to Nair et al. \cite{NEMM04} uses invariant open
covers). Here a (topological) invariant partition $\mathcal{C}_{\tau
}=\mathcal{C}(\mathcal{P},\tau,F)$ is defined by $\tau\in\mathbb{N}$, a finite
measurable partition $\mathcal{P}=\{P_{1},\ldots,P_{q}\}$ of $Q$ and
$F:\mathcal{P}\rightarrow\Omega^{\tau}$ such that%
\[
\varphi(k,P,F(P))\subset Q\text{ for }k=1,\ldots,\tau.
\]
Thus, in contrast to $(Q,\eta)$-invariant partitions (cf. Definition
\ref{Definition_inv_part}), it is required that every $x\in P$ remains in $Q$
under the feedback $F(P)$. Then a word $a=[a_{0},a_{1},\ldots,a_{n-1}]$ is
called admissible if there exists a point $x\in Q$ with $\varphi(i\tau
,x,u_{a})\in P_{a_{j}}$ for $i=0,1,\ldots,n-1$. Write $\#\mathcal{W}%
_{n}(\mathcal{C}_{\tau})$ for the number of elements in the set $\mathcal{W}%
_{n}(\mathcal{C}_{\tau})$ of all admissible words of length $n$ and define the
entropy of $\mathcal{C}_{\tau}$ by
\[
h_{top}(\mathcal{C}_{\tau}):=\lim_{n\rightarrow\infty}\frac{\log
\#\mathcal{W}_{n}(\mathcal{C}_{\tau})}{n\tau}=\inf_{n\in\mathbb{N}}\frac
{\log\#\mathcal{W}_{n}(\mathcal{C}_{\tau})}{n\tau}.
\]
A topological invariant partition of $Q$ is also a $(Q,\eta)$-partition and an
$\eta$-admissible word is also admissible in the topological sense. The
following characterization of topological invariance entropy is given in Kawan
\cite[Theorem 2.3 and its proof]{Kawa13}.

\begin{theorem}
\label{Kawan1}For a compact and controlled invariant set $Q$ it holds that%
\[
h_{inv}(Q)=\inf_{\mathcal{C}_{\tau}}h(\mathcal{C}_{\tau})=\underset
{\tau\rightarrow\infty}{\lim}\inf_{\mathcal{C}_{\tau}}h(\mathcal{C}_{\tau}),
\]
where the first infimum is taken over all invariant $Q$-partitions
$\mathcal{C}_{\tau}$ and the second infimum is taken over all invariant
$Q$-partitions $\mathcal{C}_{\tau}$ with fixed $\tau\in\mathbb{N}$.
\end{theorem}

The following theorem relates metric and topological invariance entropy.

\begin{theorem}
\label{comparison}Let $Q$ be a compact and controlled invariant set $Q$. Then
for every quasi-stationary measure $\eta$ on $Q$ the metric entropy with
respect to $\mu=\nu^{\mathbb{N}}\times\eta$ satisfies%
\[
h_{\mu}(Q)\leq h_{inv}(Q).
\]

\end{theorem}

\begin{proof}
Let $\eta$ be a quasi-stationary measure on $Q$ and fix a topological
invariant partition $\mathcal{C}_{\tau}(\mathcal{P},F)$. Then $\mathcal{C}%
_{\tau}(\mathcal{P},F)$ is also a $(Q,\eta)$-invariant partition and
$\#\mathfrak{A}_{n}(\mathcal{C}_{\tau})\leq\#\mathcal{W}_{N}(\mathcal{C}%
_{\tau})$ for every $n\in\mathbb{N}$, since an $\eta$-admissible word is also
trivially admissible in the topological sense. Using%
\[
H_{\rho^{-(n-1)\tau}\mu}(\mathfrak{A}_{n}(\mathcal{C}_{\tau}))\leq
\log\#\mathfrak{A}_{n}(\mathcal{C}_{\tau})
\]
one finds that%
\[
h_{\mu}(\mathcal{C}_{\tau},Q)=\underset{n\rightarrow\infty}{\lim\sup}\frac
{1}{n\tau}H_{\rho^{-(n-1)\tau}\mu}(\mathfrak{A}_{n}(\mathcal{C}_{\tau}%
))\leq\lim_{n\rightarrow\infty}\frac{\log\#\mathcal{W}_{n}(\mathcal{C}_{\tau
})}{n\tau}=h(\mathcal{C}_{\tau}).
\]
This yields the assertion $h_{\mu}(Q)\leq h_{inv}(Q)$.
\end{proof}

The following example illustrates the existence of quasi-stationary measures
in a simple situation (it is a modification of Colonius, Homburg and Kliemann
\cite[Example 1]{ColoHK10}).

\begin{example}
\label{Example1}Consider the family of control systems depending on a real
parameter given by $f_{\alpha}:{\mathbb{R}}/{\mathbb{Z}}\times\lbrack
-1,1]\rightarrow{\mathbb{R}}/{\mathbb{Z}}$,
\begin{equation}
f_{\alpha}(x,\omega)=x+\sigma\cos(2\pi x)+A\omega+\alpha\mod 1. \label{Ex1}%
\end{equation}
Suppose that the amplitudes $A$ and $\sigma$ as well as $\alpha$ take on small
positive values. Let a probability measure $\nu$ on $\Omega:=[-1,1]$ be given.
One obtains Markov transition probabilities%
\[
p(x,B)=\nu\{\omega\in\Omega\left\vert f_{\alpha}(x,\omega)\in B\right.
\}\text{ for }x\in{\mathbb{R}}/{\mathbb{Z}}\text{ and }B\subset{\mathbb{R}%
}/{\mathbb{Z}}.
\]
For $\alpha_{0}=\sigma-A$ the extremal graph $f_{\alpha_{0}}(\cdot,1)$ is
tangent to the diagonal at a point $c_{0}$. Now let $\alpha>\alpha_{0}$ and
consider $Q=[0.2,0.5]$. Colonius \cite[Theorem 2.9]{Colo16a} implies the
existence of a quasi-stationary measure $\eta$ for $Q$ with $0<\rho<1$ if
$\nu$ has a density with respect to Lebesgue measure and there is $\gamma>0$
such that $p(x,Q)\geq\gamma>0$ for all $x\in Q$. If we take the uniform
distribution on $\Omega=[-1,1]$, these condition are satisfied. By
\cite[Proposition 2.4]{Colo16a}, the support of the corresponding
conditionally invariant measure $\mu=\nu^{\mathbb{N}}\times\eta$ is contained
in%
\[
\{(u,x)\in\mathcal{U}\times Q\left\vert S_{Q}^{-n}(u,x)\cap\left(
\mathcal{U}\times Q\right)  \not =\emptyset\text{ for all }n\in\mathbb{N}%
\right.  \}.
\]
Let $d(\alpha)<0.5$ be given by the intersection of the lower sinusoidal curve
$f_{\alpha}(\cdot,-1)$ with the diagonal. Then points to the left of
$[d(\alpha),0.5]$ leave $Q$ backwards in time, hence they cannot be in the
support of $\eta$. Thus the quasi-stationary measure $\eta$ has support
contained in $[d(\alpha),0.5]$. Observe that for the uniform distribution on
$\Omega$ there is no stationary measure $\eta$ with support in $Q$, hence
there is no invariant measure $\mu$ of the form $\mu=\nu^{\mathbb{N}}%
\times\eta$.
\end{example}

The Variational Principle for dynamical systems states that the supremum of
the metric entropies coincides with the topological entropy. Certainly, an
analogous result for invariance entropy would give considerable structural
insight. Apart from this, however, it would only be of limited interest. The
metric invariance entropy is introduced since it is smaller than the
topological invariance entropy (see also the discussion of coder-controllers
in Section \ref{Section_Coder}). Instead of asking for measures $\mu$
maximizing the entropy one should instead look for measures minimizing
$h_{\mu}$ over a class $\mathcal{M}_{ad}$ of admissible measures, hence to
determine measures $\mu_{0}$ with minimal invariance entropy, i.e.,%
\[
h_{\mu_{0}}(Q)=\inf\nolimits_{\mu\in\mathcal{M}_{ad}}h_{\mu}(Q).
\]
This induces the question over which class $\mathcal{M}_{ad}$ of measures one
should minimize. The set of all measures $\mu=\nu^{\mathbb{N}}\times\eta$
where $\eta$ is quasi-stationary with respect to $\nu$ is too big, since one
would often obtain that the minimum is zero:\ This is illustrated by Example
\ref{Example1}. There are many stationary measures $S$ with support in $Q$:
Take $\mu=\left(  \delta_{\omega}\right)  ^{\mathbb{N}}\times\delta_{x(u)}$,
where $\delta_{\omega}$ and $\delta_{x(\omega)}$ are Dirac measures with
$f(x(\omega),\omega)=x(\omega)\in Q$ and $h_{\mu}=0$. The reason is that by
the choice $\nu=\delta_{\omega}$ the invariance problem is already solved.
Instead it seems reasonable to require at least that the support of $\nu$
coincides with the control range $\Omega$, since otherwise we would already
know that we do not need controls in $\Omega\setminus\mathrm{supp}\nu$. If
$\Omega\subset\mathbb{R}^{m}$ has positive Lebesgue measure $\lambda(\Omega)$
one might even restrict attention to the measures $\nu$ which are absolutely
continuous with respect to $\lambda$. Concerning the support of the
quasi-stationary measure $\eta$ it is immediately clear that $Q\setminus
\mathrm{supp}\eta$ does not contribute to $h_{\mu}$. Some further results will
be given in Sections \ref{Section3} and \ref{Section4}.

\section{Relations to coder-controllers\label{Section_Coder}}

This section defines coder-controllers associated with a quasi-stationary
measure rendering $Q$ invariant, and shows that their minimal entropy
coincides with the entropy\textbf{\ }$h_{\mu}(Q)$.

A coder-controller (cf., e.g., Kawan \cite[Section 2.5]{Kawa13}) may be
defined as a quadruple $\mathcal{H}=(S,\gamma,\delta,\tau)$ where
$S=(S_{k})_{k\in\mathbb{N}}$ denotes finite coding alphabets and the coder
mapping $\gamma_{k}:M^{k+1}\rightarrow S_{k}$ associates to the present and
past states the symbol $s_{k}\in S_{k}$. At time $k\tau$, the controller has
$k+1$ symbols $s_{0},\ldots,s_{k}$ available and generates a control $u_{k}%
\in\Omega^{\tau}$. The corresponding controller mapping is $\delta_{k}%
:S_{0}\times\cdots\times S_{k}\rightarrow\Omega^{\tau}$. Thus of interest are
for $x_{0}\in Q$ the sequences%
\begin{equation}
x_{k+1}:=\varphi(\tau,x_{k},u_{k}),k\in\mathbb{N}, \label{x_k}%
\end{equation}
with%
\begin{equation}
u_{k}=\delta_{k}(\gamma_{0}(x_{0}),\gamma_{1}(x_{0},x_{1}),\ldots,\gamma
_{k}(x_{0},x_{1},\ldots,x_{k}))\in\Omega^{\tau} \label{u_k}%
\end{equation}
satisfying%
\begin{equation}
\varphi(i,x_{k},u_{k})\in Q\text{ for all }i\in\{1,\ldots,\tau\}\text{ and all
}k\in\mathbb{N}. \label{3.3}%
\end{equation}
In the following construction, we suppose that the coding alphabet $S$ is
independent of $k$. Hence%
\[
\gamma_{k}:M^{k+1}\rightarrow S\text{ and }\delta_{k}:S^{k+1}\rightarrow
\Omega^{\tau}.
\]
Each sequence of symbols in $S^{k}$ defines a coding region in $Q$ which is
defined as the set of all initial states $x_{0}$ which force the coder to
generate this sequence. More precisely, for $s=(s_{0},\ldots,s_{k-1})\in
S^{k}$ let%
\[
P_{s}:=\{x_{0}\in Q\left\vert \gamma_{0}(x_{0})=s_{0},\ldots,\gamma
_{k-1}(x_{0},x_{1},\ldots,x_{k-1})=s_{k-1}\right.  \},
\]
where $x_{j},j=1,\ldots,k-1$, are generated by (\ref{x_k}), (\ref{u_k}).
Furthermore, let%
\[
u_{s}:=(\delta_{0}(s_{0}),\delta_{1}(s_{0},s_{1}),\ldots,\delta_{k-1}%
(s_{0},\ldots,s_{k-1}))\in\Omega^{k\tau}.
\]
Thus $P_{s}$ consists of the points $x_{0}$ satisfying for $j=0,\ldots,k-1$%
\[
(x_{0},\varphi(\tau,x_{0},u_{s}),\ldots,\varphi(j\tau,x_{0},u_{s}))\in
\gamma_{j}^{-1}(s_{j}),
\]
and one has $\varphi(i,x_{0},u_{s})\in Q$ for all $i=0,\ldots,k\tau$.

Again, suppose that a quasi-stationary measure $\eta$ with $\rho=\int
_{Q}p(x,Q)\eta(dx)$ corresponding to $\nu$ on $\Omega$ is fixed and let
$\mu=\nu^{\mathbb{N}}\times\eta$. The set $S_{ad}^{k}$ of $(Q,\eta
)$-admissible words\ is the set of $s\in S^{k}$ with $\eta(P_{s})>0$.

For a $(Q,\eta)$-admissible word we consider all controls $u$ such that the
application of $u$ yields the same word and renders $Q$ invariant with
probability $1$. More explicitly, define for $s_{0}\in S$
\[
A(P_{s_{0}}):=\{u\in\mathcal{U}\left\vert \varphi(i,x,u)\in Q\text{ for
}i=1,\ldots,\tau\text{ and }\eta\text{-a.a. }x\in P_{s_{0}}\right.  \}\times
P_{s_{0}}%
\]
and for $s=(s_{0},\ldots,s_{k-1})\in S_{ad}^{k}$%
\[
D_{s}:=A(P_{s_{0}})\cap S^{-\tau}A(P_{s_{1}})\cap\cdots\cap S^{-(k-1)\tau
}A(P_{s_{k-1}})\subset S_{Q}^{-(k-1)\tau}(\mathcal{U}\times Q).
\]
Observe that for $(u,x_{0})\in D_{s}$ one obtains $\varphi(j\tau,x_{0},u)\in
P_{s_{j}},j=0,\ldots,k-1$. Define a measure $\lambda_{k}$ on the (finite) set
$S_{ad}^{k}$ by%
\begin{equation}
\lambda_{k}(s):=\rho^{-(k-1)\tau}\mu(D_{s}). \label{nuk}%
\end{equation}
The considered coder maps $\gamma_{k}:M^{k+1}\rightarrow S$ will only be
defined on a set of full $\eta^{k+1}$-measure and the considered controller
maps $\delta_{k}$ will only be defined on a set of full $\lambda_{k+1}%
$-measure. The following coder-controllers render $Q$ invariant.

\begin{definition}
A $(Q,\eta)$-coder-controller is a quadruple $\mathcal{H}=(S,\gamma
,\delta,\tau)$ as above such that $\left\{  P_{s}\left\vert s\in S\right.
\right\}  $ forms a partition (modulo $\eta$-null sets) of $Q$. The entropy of
$\mathcal{H}$ is defined as%
\[
R(\mathcal{H})=\underset{k\rightarrow\infty}{\lim\sup}\frac{1}{k\tau
}H_{\lambda_{k}}(S_{ad}^{k}).
\]

\end{definition}

\begin{remark}
In general, $\lambda_{k}(S_{ad}^{k})<1$, hence the measures $\lambda_{k}$ are
not probability measures. Nevertheless, it makes sense to consider the
associated entropy, cf. Remark \ref{Remark_partition} for a similar situation.
\end{remark}

\begin{remark}
Kawan \cite[p. 72 and p. 83]{Kawa13} (see also Nair, Evans, Mareels and Moran
\cite{NEMM04}) defines coder-controllers with and without $\tau$ (i.e.,
$\tau=1$) and the coder maps $\gamma_{k}$ may or may not depend on the past
symbols (in addition to the past states). Furthermore, it is usually assumed
that the size of the set of symbols may vary with $k$. Using the same set of
symbols for every $k$ amounts to requiring that $\sup_{k\in\mathbb{N}}$
$\#S_{k}<\infty$ where $\#S_{k}$ is the size of the set $S_{k}$ of symbols at
time $k\tau$. In the (topological) definition of data rates this might be
taken into account by looking at the number of symbols actually used at time
$k\tau$. In fact, $\sup_{k\in\mathbb{N}}\#S_{k}<\infty$ in the situations
considered in \cite[Theorem 2.1 and Theorem 2.4]{Kawa13}.
\end{remark}

When one wants to relate coder-controllers to the invariance entropy in
Definition \ref{Definition_main}, it may appear rather straightforward to
identify the elements of an invariant partition with the set of symbols for a
coder-controller and the feedbacks $F(P)$ with the controls generated by a
coder-controller. This is the content of the following theorem.

\begin{theorem}
Consider system (\ref{0.1}) and let $\eta$ be a quasi-stationary measure for a
closed set $Q$ corresponding to $\nu$ on $\Omega$ and denote $\mu
=\nu^{\mathbb{N}}\times\eta$. The $\mu$-invariance entropy satisfies%
\[
h_{\mu}(Q)=\underset{\tau\rightarrow\infty}{\lim\sup}\inf_{\mathcal{H}%
}R(\mathcal{H}),
\]
where the infimum is taken over all $(Q,\eta)$-coder-controllers
$\mathcal{H}=(S,\gamma,\delta,\tau)$.
\end{theorem}

\begin{proof}
(i) First we show that for every invariant $(Q,\eta)$-partition $\mathcal{C}%
_{\tau}=\mathcal{C}_{\tau}(\mathcal{P},F)$ with $\mathcal{P}=\{P_{1}%
,\ldots,P_{q}\}$ there exists a $(Q,\eta)$-coder-controller $\mathcal{H}%
=(S,\gamma,\delta,\tau)$ rendering $Q$ invariant such that $R(\mathcal{H}%
)=h_{\mu}(\mathcal{C}_{\tau})$. In order to construct $\mathcal{H}$ let
$S:=\{1,\ldots,q\}$ and for $k\geq1$ let $S_{ad}^{k}$ be the set of
$\mathcal{C}_{\tau}$-admissible words $a$ of length $k$. Denote by
$\lambda_{k}$ the measure on $S_{ad}^{k}$ given by%
\begin{equation}
\lambda_{k}(a)=\rho^{-(k-1)\tau}\mu(D_{a}),D_{a}\in\mathfrak{A}_{k}%
(\mathcal{C}_{\tau}). \label{nu_k}%
\end{equation}
It follows that%
\[
H_{\lambda_{k}}(S_{ad}^{k})=H_{\rho^{-(k-1)\tau}\mu}(\mathfrak{A}%
_{k}(\mathcal{C}_{\tau})).
\]
The coders $\gamma_{k}:Q^{k+1}\rightarrow S$ are defined by%
\[
\gamma_{k}(x_{0},\ldots,x_{k}):=a_{k}\text{ if }x_{k}\in P_{a_{k}}.
\]
Then the entropy associated to this coder is%
\[
\underset{k\rightarrow\infty}{\lim\sup}\frac{1}{k\tau}H_{\lambda_{k}}%
(S_{ad}^{k})=\underset{k\rightarrow\infty}{\lim\sup}\frac{1}{k\tau}%
H_{\rho^{-(k-1)\tau}\mu}(\mathfrak{A}_{k}(\mathcal{C}_{\tau}))=h_{\mu
}(\mathcal{C}_{\tau}).
\]
The controller is constructed as follows. Each set $D_{a}\in\mathfrak{A}%
_{k}(\mathcal{C}_{\tau})$ is of the form%
\[
D_{a}=A(P_{a_{0}})\cap\cdots\cap S^{-(k-1)\tau}A(P_{a_{k-1}})
\]
with $A(P_{a_{i}})\in\mathfrak{A}(\mathcal{C}_{\tau})$ for all $i$. Upon
receiving the symbol $a_{k-1}$ in addition to the previous symbols
$a_{0},\ldots,a_{k-2}$ the controller finds $a=(a_{0},\ldots,a_{k-1})$ which
indexes an element of $\mathfrak{A}_{k}(\mathcal{C}_{\tau})$. Then the
controller is given by the maps $\delta_{k-1}:S_{ad}^{k}\rightarrow
\Omega^{\tau},$%
\[
\delta_{k-1}(a_{0},\ldots,a_{k-1}):=F(P_{a_{k-1}}).
\]
By the definition this yields for the corresponding solution that
$x_{(k-1)\tau+i}\in Q,i=0,\ldots,\tau$, and hence the constructed
coder-controller renders $Q$ invariant.

Taking the infimum over all invariant $(Q,\eta)$-partitions $\mathcal{C}%
_{\tau}$, then the infimum over all $(Q,\eta)$-coder-controllers
$\mathcal{H}=(S,\gamma,\delta,\tau)$ and, finally, the limit superior for
$\tau\rightarrow\infty$ one obtains%
\[
\underset{\tau\rightarrow\infty}{\lim\sup}\inf_{\mathcal{H}}R(\mathcal{H})\leq
h_{\mu}(Q).
\]
(ii) For the converse inequality it suffices to show that for an arbitrary
$(Q,\eta)$-coder-controller $\mathcal{H}=(S,\gamma,\delta,\tau)$ there is an
invariant $(Q,\eta)$-partition $\mathcal{C}_{\tau}$ with $h_{\mu}%
(\mathcal{C}_{\tau})\allowbreak=R(\mathcal{H})$. For every $s_{0}\in S$
consider the set $P_{s_{0}}$ of all points $x_{0}\in Q$ such that
$s_{0}=\gamma_{0}(x_{0})$. Since $\mathcal{H}$ renders $Q$ invariant it
follows that for $u_{0}=\delta_{0}(s)\in\Omega^{\tau}$ (this does not denote a
Dirac measure!) one has $\varphi(i,x_{0},u_{0})\in Q$ for $i=1,\ldots,\tau$
and $\eta$-a.a. $x_{0}\in P_{s_{0}}$. Since $\mathcal{P}=\{P_{s_{0}}\left\vert
s_{0}\in S\right.  \}$ is a partition modulo $\eta$-null sets of $Q$, one
obtains that $\mathcal{C}_{\tau}=\mathcal{C}_{\tau}(\mathcal{P},F)$ with
$F(P_{s_{0}})=u_{0}=\delta_{0}(s_{0}),s_{0}\in S,$ is an invariant partition
of $Q$. Then
\[
A(P_{s_{0}})=\{u\in\Omega^{\tau}\left\vert \text{ }\varphi(i,x,u)\in Q\text{
for }i=1,\ldots,\tau\text{ and for }\eta\text{-a.a. }x\in P_{s_{0}}\right.
\}\times P_{s_{0}}%
\]
and $\mathfrak{A}(\mathcal{C}_{\tau})=\{A(P_{s_{0}})\left\vert s_{0}\in
S\right.  \}$. For $s=(s_{0},\ldots,s_{k-1})\in S_{ad}^{k}$ we have%
\[
D_{s}=A(P_{s_{0}})\cap S^{-\tau}A(P_{s_{1}})\cap\cdots\cap S^{-(k-1)\tau
}A(P_{s_{k-1}}).
\]
Then the probability measure $\lambda_{k}$ on $S_{ad}^{k}$ satisfies
$\lambda_{k}(s)=\rho^{-(k-1)\tau}\mu(D_{s})$ and hence%
\[
H_{\lambda_{k}}(S_{ad}^{k})=H_{\rho^{-(k-1)\tau}\mu}(\mathfrak{A}%
_{k}(\mathcal{C}_{\tau})).
\]
Thus one obtains for the entropy%
\[
R(\mathcal{H})=\underset{k\rightarrow\infty}{\lim\sup}\frac{1}{k\tau
}H_{\lambda_{k}}(S_{ad}^{k})=\underset{k\rightarrow\infty}{\lim\sup}\frac
{1}{k\tau}H_{\rho^{-(k-1)\tau}\mu}(\mathfrak{A}_{k}(\mathcal{C}_{\tau
}))=h_{\mu}(\mathcal{C}_{\tau}).
\]

\end{proof}

\section{Invariance entropy and relative invariance\label{Section3}}

In this section and in Section \ref{Section4}, we analyze when the metric
invariance entropy of $Q$ is already determined on certain subsets of $Q$. The
analysis is based on the following relative invariance property.

\begin{definition}
Consider for system (\ref{0.1}) subsets $K\subset Q\subset M$. The set $K$ is
called invariant in $Q$, if $x\in K\,$and $f(x,\omega)\not \in K$ for some
$\omega\in\Omega$ implies $f(x,\omega)\not \in Q$.
\end{definition}

Thus a solution $\varphi(\cdot,x,u)$ can leave a set $K$ which is invariant in
$Q$ only if it also leaves $Q.$

For (measurable) subsets $K$ which are invariant in $Q$ and a quasi-stationary
measure $\eta$ on $Q$ we define invariant $(K,\eta)$-partitions as in
Definition \ref{Definition_inv_part} with $Q$ replaced by $K$. Then one can
define the invariance entropy $h_{\mu}(K)$ of $K$ as in Definition
\ref{Definition_main}, again replacing $Q$ by $K$. In the following, objects
associated with invariant $(K,\eta)$-partitions and invariant $(Q,\eta
)$-partitions are denoted with a superscript $K$ and $Q$, respectively.

First we determine relations between invariant partitions of $K$ and $Q$. We
call a map $F$ on $Q$ nonsingular with respect to $\eta$ if $\eta(E)=0$
implies $\eta(F^{-1}E)=0$ for $E\subset Q$.

\begin{lemma}
\label{lemma_restrict_extend}Let $K$ be a closed set which is invariant in a
closed set $Q\subset M$.

(i) Then every invariant $(Q,\eta)$-partition $\mathcal{C}_{\tau}%
^{Q}=\mathcal{C}_{\tau}(\mathcal{P}^{Q},F^{Q})$ induces an invariant
$(K,\eta)$-partition $\mathcal{C}_{\tau}^{K}=\mathcal{C}_{\tau}(\mathcal{P}%
^{K},F^{K})$ given by%
\[
\mathcal{P}^{K}:=\{P\cap K\left\vert P\in\mathcal{P}^{Q}\right.  \}\text{ and
}F^{K}(P\cap K):=F^{Q}(P).
\]

(ii) Assume that there are a finite measurable cover of $Q$ by sets
$V^{1},\ldots,V^{N}$, control functions $v^{1},\ldots,v^{N}\in\mathcal{U}$ and
times $\tau^{1},\ldots,\tau^{N}\in\mathbb{N}$ such that for all $j=1,\ldots,N$
and a.a. $x\in V^{j}$%
\[
\varphi(k,x,v^{j})\in Q\text{ for }k=1,\ldots,\tau^{j}\text{ and }\varphi
(\tau^{j},x,v^{j})\in K
\]
and the maps $\varphi(\tau^{j},\cdot,v^{j})$ on $Q$ are nonsingular with
respect to $\eta$.

Then every invariant $(K,\eta)$-partition $\mathcal{C}_{\tau}^{K}%
=\mathcal{C}_{\tau}(\mathcal{P}^{K},F^{K})$ with $\tau\geq\bar{\tau}%
:=\max\limits_{j=1,\ldots,N}\tau^{j}$ can be extended to an invariant
$(Q,\eta)$-partition $\mathcal{C}_{\tau}^{Q}=\mathcal{C}_{\tau}(\mathcal{P}%
^{Q},F^{Q})$ such that $\#\mathcal{P}^{Q}\leq(1+N)\#\mathcal{P}^{K}$,
$\mathcal{P}^{K}\subset\mathcal{P}^{Q}$ and%
\[
F^{Q}(P)=F^{K}(P)\text{ if }P\in\mathcal{P}^{K}\text{ and }\varphi(\tau
,P^{Q},F^{Q}(P^{Q}))\subset K\text{ for }P^{Q}\in\mathcal{P}^{Q}.
\]

\end{lemma}

\begin{proof}
(i) Let $P\in\mathcal{P}^{Q}$. Since $K$ is invariant in $Q$ it follows from
$P\cap K\subset K$ and $\varphi(k,x,F^{Q}(P))\in Q$ for all $k=0,\ldots,\tau$
and $\eta$-a.a. $x\in P$ that%
\[
\varphi(k,x,F^{K}(P\cap K))=\varphi(k,x,F^{Q}(P))\in K\text{ for }%
\eta\text{-a.a. }x\in P.
\]
Thus $\mathcal{C}_{\tau}(\mathcal{P}^{K},F^{K})$ is an invariant $(K,\eta)$-partition.

(ii) Let $\mathcal{C}_{\tau}^{K}=\mathcal{C}_{\tau}(\mathcal{P}^{K},F^{K})$ be
an invariant $K$-partition with $\tau\geq\bar{\tau}$. The cover of $Q$ induces
a finite partition $\mathcal{P}_{1}$ of $Q\setminus K$ such that for every
$P_{1}^{j}\in\mathcal{P}_{1}$ the control $F_{1}(P_{1}^{j}):=(v_{0}^{j}%
,\ldots,v_{\tau^{j}-1}^{j})\in\Omega^{\tau^{j}}$ satisfies for a.a. $x\in
P_{1}^{j}$ one has $\varphi(k,x,F_{1}(P_{1}^{j}))\in Q$ for all $k=0,\ldots
,\tau^{j}$. and $\varphi(\tau^{j},x,F_{1}(P_{1}^{j}))\in K$. In fact, we
obtain a partition of $Q\setminus K$ by defining $P_{1}^{1}:=(Q\setminus
K)\cap V^{1}$ and%
\[
P_{1}^{j}:=\left[  (Q\setminus K)\cap V^{j}\right]  \setminus\bigcup
\nolimits_{i<j}P_{1}^{i}\text{ for }j>1.
\]
Then $P_{1}^{j}\subset V^{j}$ implying for a.a. $x\in P_{1}^{j}$ that
$\varphi(k,x,F_{1}(P_{1}^{j}))\in Q$ for $k=1,\ldots,\tau^{j}$ and
$\varphi(\tau^{j},x,F_{1}(P_{1}^{j}))\in P^{i}\subset K$ for some $P^{i}%
\in\mathcal{P}^{K}$. By nonsingularity with respect to $\eta$ of $\varphi
(\tau^{j},\cdot,v^{j})=\varphi(\tau^{j},\cdot,F_{1}(P_{1}^{j}))$ it follows
that for $\eta$-a.a. $x\in P_{1}^{j}$ there is $P^{i}\in\mathcal{P}^{K}$ such
that for all $k=0,\ldots,\tau$%
\[
\varphi(k,\varphi(\tau^{j},x,F_{1}(P_{1}^{j})),F^{K}(P^{i}))\in K\text{.}%
\]
Define an invariant $Q$-partition $\mathcal{C}_{\tau}^{Q}=\mathcal{C}_{\tau
}(\mathcal{P}^{Q},F^{Q})$ in the following way: The partition consists of the
sets in $\mathcal{P}^{K}$ together with all (nonvoid) sets of the form%
\[
P^{i,j}:=\left\{  x\in P_{1}^{j}\left\vert \varphi(\tau^{j},x,F_{1}(P_{1}%
^{j}))\in P^{i}\right.  \right\}
\]
with feedbacks defined as follows: Let
\[
F^{Q}(P^{i,j})_{k}:=\left\{
\begin{array}
[c]{ccc}%
v_{k}^{j} & \text{for} & k=0,\ldots,\tau^{j}-1\\
u_{k-\tau^{j}}^{i} & \text{for} & k=\tau^{j},\ldots,\tau-1.
\end{array}
\right.
\]
and for $P^{i}\in\mathcal{P}^{K}$ let%
\[
F^{Q}(P^{i}):=F^{K}(P^{i})=(u_{0}^{i},\ldots,u_{\tau-1}^{i}).
\]
This is well defined, since $\tau-\tau^{j}\geq0$ (we use only the first part
of $F^{K}(P^{i})$) and hence $\mathcal{C}_{\tau}(\mathcal{P}^{Q},F^{Q}) $ is
an invariant $Q$-partition with $\#\mathcal{P}^{Q}\leq\#\mathcal{P}%
^{K}+\left(  \#\mathcal{P}^{K}\cdot\#\mathcal{P}_{1}\right)  \leq
(1+N)\#\mathcal{P}^{K}$.
\end{proof}

The following theorem shows when the invariance entropy of $Q$ is already
determined on an subset $K$ that is invariant in $Q$.

\begin{theorem}
\label{TheoremA_alt}Consider control system (\ref{0.1}). Let $K$ be a closed
invariant subset in $Q$, fix a quasi-stationary measure $\eta$ on $Q$ for a
probability measure $\nu$ on the control range $\Omega$ and let $\mu
=\nu^{\mathbb{N}}\times\eta$.

(i) Then the invariance entropy of $K$ is bounded above by the invariance
entropy of $Q$, $h_{\mu}(K)\leq h_{\mu}(Q)$.

(ii) Suppose that there are a finite measurable cover of $Q$ by sets
$V^{1},\ldots,V^{N}$, control functions $v^{1},\ldots,v^{N}\in\mathcal{U}$ and
times $\tau^{1},\ldots,\tau^{N}\in\mathbb{N}$ such that for all $j=1,\ldots,N$
and a.a. $x\in V^{j}$%
\[
\varphi(k,x,v^{j})\in Q\text{ for }k=1,\ldots,\tau^{j}\text{ and }\varphi
(\tau^{j},x,v^{j})\in K,
\]
and the maps $\varphi(\tau^{j},\cdot,v^{j})$ on $Q$ are nonsingular with
respect to $\eta$.

Then $h_{\mu}(Q)=h_{\mu}(K)$ follows.

(iii) If the assumptions in (ii) are satisfied and $K$ is the disjoint union
of sets $K_{1},\ldots,K_{m}$ which are closed and invariant in $Q$, then%
\[
\max\nolimits_{i}h_{\mu}(K_{i})\leq h_{\mu}(Q)\leq h_{\mu}(K_{1}%
)+\cdots+h_{\mu}(K_{m}).
\]

\end{theorem}

\begin{proof}
(i) Let $\mathcal{C}_{\tau}^{Q}=\mathcal{C}_{\tau}(\mathcal{P}^{Q},F^{Q})$ be
an invariant $(Q,\eta)$-partition and consider the induced invariant
$(K,\eta)$-partition $\mathcal{C}_{\tau}^{K}=\mathcal{C}_{\tau}(\mathcal{P}%
^{K},F^{K})$ according to Lemma \ref{lemma_restrict_extend}(i). We will show
that $h_{\mu}(\mathcal{C}_{\tau}^{K})\leq h_{\mu}(\mathcal{C}_{\tau}^{Q}) $.
Then, taking first the infimum over all invariant $Q$-partitions
$\mathcal{C}_{\tau}^{Q}$ and then over all invariant $K$-partitions
$\mathcal{C}_{\tau}^{K}$ and, finally, the limit superior for $\tau
\rightarrow\infty$, one concludes, as claimed, that $h_{\mu}(K)\leq h_{\mu
}(Q)$.

For $P\in\mathcal{P}^{Q}$ consider $v\in\mathcal{U}$ with $\varphi(i,x,v)\in
Q$ for all $i=1,\ldots,\tau$ and $\eta$-a.a. $x\in P$. If $x$ is even in
$P\cap K$, then invariance of $K$ in $Q$ implies that $\varphi(i,x,u)\in K$
for all $i=1,\ldots,\tau$. It follows for all $P_{K}=P\cap K\in\mathcal{P}%
^{K}$ that%
\[
A(P\cap K,\mathcal{C}_{\tau}^{K})=A(P,\mathcal{C}_{\tau}^{Q})\cap\left(
\mathcal{U}\times K\right)
\]
showing that $\mathfrak{A}^{K}=\mathfrak{A}^{Q}\cap\left(  \mathcal{U}\times
K\right)  $.

For a $\mathcal{C}_{\tau}^{K}$-admissible partition sequence corresponding to
a word $a$ abbreviate $A_{a_{i}}^{K}=A(P_{a_{i}}\cap K,\mathcal{C}_{\tau}%
^{K})$ with $P_{a_{i}}\in\mathcal{P}$. Then $(P_{a_{0}},\ldots,P_{a_{n-1}})$
is a $\mathcal{C}_{\tau}^{Q}$-admissible partition sequence and hence a
$\mathcal{C}_{\tau}^{K}$-admissible word $a$ is also $\mathcal{C}_{\tau}^{Q}%
$-admissible. Consider%
\[
D_{a}^{K}=A_{a_{0}}^{K}\cap S^{-\tau}A_{a_{1}}^{K}\cap\cdots\cap
S^{-(n-1)\tau}A_{a_{n-1}}^{K}.
\]
The corresponding $\mathcal{C}_{\tau}^{Q}$-admissible set%
\[
D_{a}^{Q}=A_{a_{0}}^{Q}\cap S^{-\tau}A_{a_{1}}^{Q}\cap\cdots\cap
S^{-(n-1)\tau}A_{a_{n-1}}^{Q}%
\]
satisfies $D_{a}^{Q}\cap\left(  \mathcal{U}\times K\right)  =D_{a}^{K}$, since
$(u,x)\in A_{a_{0}}^{Q}\cap\left(  \mathcal{U}\times K\right)  $ satisfies
$(\theta^{i\tau}u,\varphi(i\tau,x,u)\in A_{a_{i}}^{K},i=0,\ldots,(n-1)\tau$.
This shows that
\begin{equation}
\mathfrak{A}_{n}^{K}\subset\mathfrak{A}_{n}^{Q}\cap(\mathcal{U}\times K)\text{
for all }n, \label{A_K_N}%
\end{equation}
where at the right hand side consists of the elements of $\mathfrak{A}_{n}%
^{Q}$ intersected with $\mathcal{U}\times K$. In order to compute the entropy,
first consider $D\in\mathfrak{A}_{n}^{Q}$ with $\rho^{-(n-1)\tau}\mu
(D)\leq1/e$. Then also $\rho^{-(n-1)\tau}\mu(D\cap\left(  \mathcal{U}\times
K\right)  )\leq1/e$ and it follows that%
\begin{equation}
\phi\left(  \rho^{-(n-1)\tau}\mu(D\right)  )\leq\phi\left(  \rho^{-(n-1)\tau
}\mu(D\cap\left(  \mathcal{U}\times K\right)  )\right)  , \label{small}%
\end{equation}
since $\phi$ is monotonically decreasing on $[0,1/e]$. For every
$n\in\mathbb{N}$ there are at most three sets $D\in\mathfrak{A}_{n}^{Q}$ with
$\rho^{-(n-1)\tau}\mu\left(  D\right)  \geq1/e$, since they are disjoint and
the sum of the measures of four mutually disjoint sets $D\subset
\mathcal{A}_{n}^{Q}\subset S_{Q}^{-(n-1)\tau}(\mathcal{U}\times Q)$ would be
greater than or equal to $4/e>1=\rho^{-(n-1)\tau}\mu(S_{Q}^{-(n-1)\tau
}(\mathcal{U}\times Q))$. Let%
\[
\mathfrak{A}_{n}^{Q,big}:=\{D\in\mathfrak{A}_{n}^{Q}\left\vert \rho
^{-(n-1)\tau}\mu\left(  D\right)  \geq1/e\right.  \}.
\]
Then $\#\mathfrak{A}_{n}^{Q,big}\leq3$ and, using $\phi(x)\geq\phi
(1/e)=-1/e,x\in\lbrack0,1]$, it follows that%
\[
\sum_{D\in\mathfrak{A}_{n}^{Q,big}}\phi\left(  \rho^{-(n-1)\tau}\mu
(D\cap\left(  \mathcal{U}\times K\right)  )\right)  \geq-3/e.
\]
We find, using $\phi(x)\leq0$ and (\ref{small}), (\ref{A_K_N}),%
\begin{align*}
-  &  H_{\rho^{-(n-1)\tau}\mu}\left(  \mathfrak{A}_{n}^{Q}\right) \\
&  =\sum_{D\in\mathfrak{A}_{n}^{Q}\setminus\mathfrak{A}_{n}^{Q,big}}%
\phi\left(  \rho^{-(n-1)\tau}\mu(D)\right)  +\sum_{D\in\mathfrak{A}%
_{n}^{Q,big}}\phi\left(  \rho^{-(n-1)\tau}\mu(D)\right) \\
&  \leq\sum_{D\in\mathfrak{A}_{n}^{Q}\setminus\mathfrak{A}_{n}^{Q,big}}%
\phi\left(  \rho^{-(n-1)\tau}\mu(D\cap\left(  \mathcal{U}\times K\right)
)\right) \\
&  =\sum_{D\in\mathfrak{A}_{n}^{Q}}\phi\left(  \rho^{-(n-1)\tau}\mu
(D\cap\left(  \mathcal{U}\times K\right)  )\right)  -\sum_{D\in\mathfrak{A}%
_{n}^{Q,big}}\phi\left(  \rho^{-(n-1)\tau}\mu(D\cap\left(  \mathcal{U}\times
K\right)  )\right) \\
&  \leq\sum_{D^{K}\in\mathfrak{A}_{n}^{K}}\phi\left(  \rho^{-(n-1)\tau}%
\mu(D^{K})\right)  +3/e\\
&  \leq-H_{\rho^{-(n-1)\tau}\mu}\left(  \mathfrak{A}_{n}^{K}\right)  +3/e.
\end{align*}
This implies%
\begin{align*}
h_{\mu}(\mathcal{C}_{\tau}^{Q})  &  =\underset{n\rightarrow\infty}{\lim\sup
}\frac{1}{n\tau}H_{\rho^{-(n-1)\tau}\mu}(\mathfrak{A}_{n}(\mathcal{C}_{\tau
}^{Q}))\geq\underset{n\rightarrow\infty}{\lim\sup}\frac{1}{n\tau}\left[
H_{\rho^{-(n-1)\tau}\mu}(\mathfrak{A}_{n}(\mathcal{C}_{\tau}^{K}))-3/e\right]
\\
&  =h_{\mu}(\mathcal{C}_{\tau}^{K}).
\end{align*}

(ii) By Lemma \ref{lemma_restrict_extend} we can extend an invariant
$(K,\eta)$-partition $\mathcal{C}_{\tau}^{K}=\mathcal{C}_{\tau}(\mathcal{P}%
^{K},F^{K})$ with $\tau>\bar{\tau}:=\max_{j=1,\ldots,N}\tau^{j}$ to an
invariant $(Q,\eta)$-partition $\mathcal{C}_{\tau}^{Q}=\mathcal{C}_{\tau
}(\mathcal{P}^{Q},F^{Q})$. We claim that $h_{\mu}(\mathcal{C}_{\tau}^{Q})\leq
h_{\mu}(\mathcal{C}_{\tau}^{K})$. Then assertion (ii) will follow, if we take
first the infimum over all invariant $(K,\eta)$-partitions $\mathcal{C}_{\tau
}^{K}$, then the infimum over all invariant $(Q,\eta)$-partitions
$\mathcal{C}_{\tau}^{Q}$ and, finally, let $\tau\rightarrow\infty$.

Recall from Lemma \ref{lemma_restrict_extend} that $\mathcal{P}^{K}%
\subset\mathcal{P}^{Q},~F^{K}(P)=F^{Q}(P)$ for $P\in\mathcal{P}^{K}$ and
$\varphi(\tau,x,F^{Q}(P))\in K$ for $\eta$-a.a. $x\in P$ and all
$P\in\mathcal{P}^{Q}$. It follows for all $P\in\mathcal{P}^{K}\subset
\mathcal{P}^{Q}$ that%
\[
A^{K}(P)=A^{Q}(P)\text{ and hence }\mathfrak{A}^{K}:=\mathfrak{A}%
(\mathcal{C}_{\tau}^{K})\subset\mathfrak{A}^{Q}:=\mathfrak{A}(\mathcal{C}%
_{\tau}^{Q}).
\]
Let $(P_{0},\ldots,P_{n-1})$ be a $\mathcal{C}_{\tau}^{Q}$-admissible
partition sequence in $\mathcal{P}^{Q}$. Since $K$ is invariant in $Q$, it
follows that $P_{i}\subset K$ for all $i=1,\ldots,n-1$, and hence for every
$k\geq\tau$ the controls are given by feedbacks $F^{K}$ keeping the system in
$K$. Hence $(P_{1},\ldots,P_{n-1})$ is a $\mathcal{C}^{K}$-admissible sequence
in $\mathcal{P}^{K}$.

Together, this implies for a $\mathcal{C}_{\tau}^{Q}$-admissible word
$a^{n+1}=[a_{0},a_{1},\ldots,a_{n}]$ of length $n+1$ and for the elements%
\[
D_{a^{n+1}}=A_{a_{0}}\cap S^{-\tau}A_{a_{1}}\cap\cdots\cap S^{-n\tau}A_{a_{n}%
}\in\mathfrak{A}_{n+1}^{Q},
\]
that $a^{n}:=[a_{1},\ldots,a_{n}]$ is a $\mathcal{C}^{K}$-admissible word with%
\[
D_{a^{n}}=A_{a_{1}}\cap\cdots\cap S^{-(n-1)\tau}A_{a_{n}}\in\mathfrak{A}%
_{n}^{K}.
\]
Hence $D_{a^{n+1}}=A_{a_{0}}\cap S^{-\tau}D_{a^{n}}$ with $A_{a_{0}}%
\in\mathfrak{A}^{Q}$ and we find%
\begin{equation}
\mu\left(  D_{a^{n+1}}\right)  =\mu\left(  A_{a_{0}}\cap S^{-\tau}D_{a^{n}%
}\right)  \leq\mu\left(  S_{Q}^{-\tau}D_{a^{n}}\right)  =\rho^{\tau}\mu\left(
D_{a^{n}}\right)  . \label{n+1}%
\end{equation}
Define%
\begin{align*}
\mathfrak{A}_{n}^{K,big}  &  :=\{D_{a^{n}}\in\mathfrak{A}_{n}^{K}\left\vert
\rho^{-(n-1)\tau}\mu\left(  D_{a^{n}}\right)  \geq1/e\right.  \},\\
\mathfrak{A}_{n+1}^{Q,big}  &  :=\{D_{a^{n+1}}=A_{a_{0}}\cap S^{-\tau}%
D_{a^{n}}\in\mathfrak{A}_{n+1}^{Q}\left\vert A_{a_{0}}\in\mathfrak{A}%
^{Q}\text{ and }D_{a^{n}}\in\mathfrak{A}_{n}^{K,big}\right.  \}.
\end{align*}
Then, as above, $\#\mathfrak{A}_{n}^{K,big}\leq3$. Furthermore, by Lemma
\ref{lemma_restrict_extend}, the number of elements of $\mathfrak{A}%
_{n+1}^{Q,big}$ is bounded, independently of $n$, by
\[
\#\mathfrak{A}_{n+1}^{Q,big}\leq3\cdot\#\mathfrak{A}^{Q}\leq3(N+1)\cdot
\#\mathcal{P}^{K}.
\]
For $D_{a^{n+1}}=A_{a_{0}}\cap S^{-\tau}D_{a^{n}}\in\mathfrak{A}_{n+1}^{Q}$
with $D_{a^{n}}\in\mathfrak{A}_{n}^{K}\setminus\mathfrak{A}_{n}^{K,big}$, it
follows from (\ref{n+1}) that%
\[
\rho^{-n\tau}\mu\left(  D_{a^{n+1}}\right)  \leq\rho^{-n\tau}\rho^{\tau}%
\mu\left(  D_{a^{n}}\right)  =\rho^{-(n-1)\tau}\mu\left(  D_{a^{n}}\right)
\leq1/e,
\]
and hence, using monotonicity of $\phi$ on $[0,1/e]$, for all $D_{a^{n+1}}%
\in\mathfrak{A}_{n+1}^{Q}\setminus\mathfrak{A}_{n+1}^{Q,big}$%
\[
\phi\left(  \rho^{-n\tau}\mu\left(  D_{a^{n+1}}\right)  \right)  \geq
\phi\left(  \rho^{-(n-1)\tau}\mu(D_{a^{n}})\right)  .
\]
It follows that%
\begin{align*}
H_{\rho^{-n\tau}\mu}\left(  \mathfrak{A}_{n+1}^{Q}\right)   &  =-\sum
_{D\in\mathfrak{A}_{n+1}^{Q}}\phi\left(  \rho^{-n\tau}\mu(D)\right) \\
&  \leq-\sum_{D\in\mathfrak{A}_{n+1}^{Q}\setminus\mathfrak{A}_{n+1}^{Q,big}%
}\phi\left(  \rho^{-n\tau}\mu(D)\right)  +3(N+1)\cdot\#\mathcal{P}^{K}/e\\
&  \leq-\sum_{D\in\mathfrak{A}_{n}^{K}\setminus\mathfrak{A}_{n}^{K,big}}%
\phi\left(  \rho^{-(n-1)\tau}\mu(D)\right)  +3(N+1)\cdot\#\mathcal{P}^{K}/e\\
&  \leq-\sum_{D\in\mathfrak{A}_{n}^{K}}\phi\left(  \rho^{-(n-1)\tau}%
\mu(D)\right)  +3(N+1)\cdot\#\mathcal{P}^{K}/e\\
&  =H_{\rho^{-(n-1)\tau}\mu}\left(  \mathfrak{A}_{n}^{K}\right)
+3(N+1)\cdot\#\mathcal{P}^{K}/e.
\end{align*}
We conclude that, as claimed,%
\begin{align*}
h_{\mu}(\mathcal{C}_{\tau}^{Q})  &  =\underset{n\rightarrow\infty}{\lim\sup
}\frac{1}{n\tau}H_{\rho^{-(n-1)\tau}\mu}(\mathfrak{A}_{n}(\mathcal{C}_{\tau
}^{Q}))\\
&  =\underset{n\rightarrow\infty}{\lim\sup}\frac{1}{n\tau}H_{\rho^{-n\tau}\mu
}(\mathfrak{A}_{n+1}(\mathcal{C}_{\tau}^{Q}))\\
&  \leq\underset{n\rightarrow\infty}{\lim\sup}\frac{1}{n\tau}\left[
H_{\rho^{-(n-1)\tau}\mu}(\mathfrak{A}_{n}(\mathcal{C}_{\tau}^{K}%
))+3(N+1)\cdot\#\mathcal{P}^{K}/e\right]  =h_{\mu}\left(  \mathcal{C}_{\tau
}^{K}\right)  .
\end{align*}

(iii) The first assertion follows by (i) using that each $K_{i}$ is invariant
in $Q$. For the second inequality use that by (ii) $h_{\mu}(Q)=h_{\mu}(K)$.
Take invariant partitions $\mathcal{C}_{\tau}^{K_{i}}=\mathcal{C}_{\tau
}(\mathcal{P}^{i},F^{i})$ of $K_{i}$ for each $i$. Then $\mathcal{C}_{\tau
}^{K}=\mathcal{C}_{\tau}(\mathcal{P}^{K},F^{K})$ defined by $\mathcal{P}%
^{K}=\bigcup\mathcal{P}^{i},F_{\left\vert \mathcal{P}^{i}\right.  }^{K}%
=F_{i},i=1,\ldots,m$, forms an invariant partition of $K$ and one finds%
\[
\mathfrak{A}_{n}^{K}(\mathcal{C}_{\tau})=%
{\displaystyle\bigcup\limits_{i=1}^{m}}
\mathfrak{A}_{n}^{K_{i}}(\mathcal{C}_{\tau})\text{ and }H_{\rho^{-(n-1)\tau
}\mu}(\mathfrak{A}_{n}(\mathcal{C}_{\tau}^{K}))=\sum_{i=1}^{m}H_{\rho
^{-(n-1)\tau}\mu}(\mathfrak{A}_{n}(\mathcal{C}_{\tau}^{K_{i}})).
\]
This implies that%
\begin{align*}
h_{\mu}\left(  \mathcal{C}_{\tau}^{K}\right)   &  \leq\underset{n\rightarrow
\infty}{\lim\sup}\frac{1}{n\tau}H_{\rho^{-(n-1)\tau}\mu}(\mathfrak{A}%
_{n}(\mathcal{C}_{\tau}^{K}))\\
&  =\underset{n\rightarrow\infty}{\lim\sup}\frac{1}{n\tau}\sum_{i=1}%
^{m}H_{\rho^{-(n-1)\tau}\mu}(\mathfrak{A}_{n}(\mathcal{C}_{\tau}^{K_{i}}%
))\leq\sum_{i=1}^{m}h_{\mu}\left(  \mathfrak{A}_{n}(\mathcal{C}_{\tau}^{K_{i}%
})\right)  .
\end{align*}
Now take the infimum over all invariant partitions $\mathcal{C}_{\tau}^{K_{j}%
}$ of the $K_{j}$, then over all invariant partitions of $K$ and, finally, the
limit superior for $\tau\rightarrow\infty$.
\end{proof}

\begin{remark}
\label{vanDoornPollett}If the assumptions of Theorem \ref{TheoremA_alt}(ii)
are satisfied, one might conjecture that the support of any quasi-stationary
measure is contained in $K$. However, this cannot be expected as seen from the
description of all quasi-stationary measures for finite state spaces given in
van Doorn and Pollett \cite[Theorem 4.2]{DoorP09}. A simple example is given
in Bena\"{\i}m, Cloez and Panloup \cite[Example 3.5]{BenaCP16}.
\end{remark}

For a quasi-stationary measure $\eta$ and a closed invariant subset $Q_{1}$ in
$Q$ with $\eta(Q_{1})>0$ let $\mu_{Q_{1}}:=\nu^{\mathbb{N}}\times\eta_{Q_{1}}%
$, where $\eta_{Q_{1}}(\cdot)=\eta(\cdot\cap Q_{1})/\eta(Q_{1})$ is the
conditional measure on $Q_{1}$.

\begin{corollary}
\label{Corollary_average}Consider control system (\ref{0.1}) and fix a
quasi-stationary measure $\eta$ on $Q$ for a probability measure $\nu$ on the
control range $\Omega$ and let $\mu=\nu^{\mathbb{N}}\times\eta$. Suppose that
$Q$ is the disjoint union of closed pairwise disjoint sets $Q_{i}%
,i=1,\ldots,m$, which are invariant in $Q$ with $\eta(Q_{i})>0$ for all $i$.
Then all $\eta_{Q_{i}}$ are quasi-stationary for $Q_{i}$ with the same
constant as $\eta$ and
\begin{equation}
\max_{i=1,\ldots,m}\left\{  \mu(Q_{i})h_{\mu_{Q_{1}}}(Q_{1})\right\}  \leq
h_{\mu}(Q)\leq\eta(Q_{1})h_{\mu_{Q_{1}}}(Q_{1})+\cdots+\eta(Q_{m}%
)h_{\mu_{Q_{m}}}(Q_{m}). \label{disjoint}%
\end{equation}

\end{corollary}

\begin{proof}
Abbreviate $\eta_{i}:=\eta_{Q_{i}}$ for all $i$. One has the convex
combination $\eta=\eta(Q_{1})\eta_{1}+\cdots+\eta(Q_{m})\eta_{m}(Q_{m}) $
implying for all $k\in\mathbb{N}$%
\[
\rho^{-k}\eta=\eta(Q_{1})\rho^{-k}\eta_{1}+\cdots+\eta(Q_{m})\rho^{-k}\eta
_{m}(Q_{m}).
\]
For $A\subset Q_{i}$, invariance of $Q_{j},j\not =i$, in $Q$ implies%
\begin{align*}
\rho\eta_{i}(A)  &  =\rho\frac{\eta(A)}{\eta(Q_{i})}=\frac{1}{\eta(Q_{i})}%
\int_{Q}p(x,A)\eta(dx)=\frac{1}{\eta(Q_{i})}\int_{Q_{i}}p(x,A)\eta(dx)\\
&  =\int_{Q_{i}}p(x,A)\eta_{i}(dx).
\end{align*}
Hence $\eta_{i}$ is quasi-stationary for $Q_{i}$ with the same constant $\rho$
as $\eta$. Then one easily sees that $h_{\mu}(Q_{i})=\eta(Q_{i})h_{\mu_{i}%
}(Q_{i})$. Now the first inequality in (\ref{disjoint}) follows from Theorem
\ref{TheoremA_alt}(i) and the second inequality follows from the arguments in
the proof of \ref{TheoremA_alt}(iii) if one notes that here $K=Q$.
\end{proof}

For the incremental invariance entropy we can only show the analog of Theorem
\ref{TheoremA_alt}(i). Here the arguments are a bit more involved.

\begin{theorem}
\label{TheoremA}Consider control system (\ref{0.1}). Let $K$ be a closed
invariant subset in $Q$, fix a quasi-stationary measure $\eta$ on $Q$ for a
probability measure $\nu$ on the control range $\Omega$ and let $\mu
=\nu^{\mathbb{N}}\times\eta$.

Then the incremental invariance entropy of $K$ is bounded above by the
incremental invariance entropy of $Q$, $h_{\mu}^{inc}(K)\leq h_{\mu}^{inc}(Q)$.
\end{theorem}

\begin{proof}
Let $\mathcal{C}_{\tau}^{Q}=\mathcal{C}_{\tau}(\mathcal{P}^{Q},F^{Q})$ \ be an
invariant $(Q,\eta)$-partition and consider the induced invariant $(K,\eta
)$-partition $\mathcal{C}_{\tau}^{K}=\mathcal{C}_{\tau}(\mathcal{P}^{K}%
,F^{K})$ according to Lemma \ref{lemma_restrict_extend}. We will show that
$h_{\mu}^{inc}(\mathcal{C}_{\tau}^{K})\leq h_{\mu}^{inc}(\mathcal{C}_{\tau
}^{Q})$. Then the assertion will follow.

The arguments used to prove formula (\ref{A_K_N}) also show that for a set
$D\cap\left(  \mathcal{U}\times K\right)  \in\mathfrak{A}_{n}^{K}%
=\mathfrak{A}_{n}(\mathcal{C}_{\tau}^{K})$ one obtains (cf. (\ref{A(D)}))%
\[
\mathfrak{A}_{n+1}^{K}(D\cap\left(  \mathcal{U}\times K\right)  )\subset
\mathfrak{A}_{n+1}^{Q}(D)\cap\left(  \mathcal{U}\times K\right)  ,
\]
and hence the unions satisfy%
\[
\mathcal{A}_{n+1}^{K}(D\cap\left(  \mathcal{U}\times K\right)  )\subset
\mathcal{A}_{n+1}^{Q}(D)\cap\left(  \mathcal{U}\times K\right)  \subset
\mathcal{A}_{n+1}^{Q}(D).
\]
Next we consider the conditional entropy (cf. (\ref{condition_fb0})),%
\begin{align}
&  -H_{\rho^{-n\tau}\mu}\left(  \mathfrak{A}_{n+1}^{K}\left\vert
\mathfrak{A}_{n}^{K,n+1})\right.  \right) \label{E1}\\
&  =\rho^{-n\tau}\sum_{D\cap\left(  \mathcal{U}\times K\right)  \in
\mathfrak{A}_{n}^{K,n+1}}\mu(\mathcal{A}_{n+1}^{K}(D\cap\left(  \mathcal{U}%
\times K\right)  )\sum_{E\in\mathfrak{A}_{n+1}^{K,}}\phi\left(  \frac
{\mu(D\cap\left(  \mathcal{U}\times K\right)  \cap E)}{\mu(\mathcal{A}%
_{n+1}^{K}(D\cap\left(  \mathcal{U}\times K\right)  ))}\right)  .\nonumber
\end{align}
Fix an element $D\cap\left(  \mathcal{U}\times K\right)  \in\mathfrak{A}%
_{n}^{K}$ and let
\begin{align*}
\alpha_{1}  &  :=\mu\left(  \mathcal{A}_{n+1}^{K}(D\cap\left(  \mathcal{U}%
\times K\right)  )\right)  ,~\alpha_{2}:=\mu\left(  \mathcal{A}_{n+1}%
^{Q}(D)\setminus\mathcal{A}_{n+1}^{K}(D\cap\left(  \mathcal{U}\times K\right)
)\right)  ,\\
\alpha &  :=\alpha_{1}+\alpha_{2}=\mu\left(  \mathcal{A}_{n+1}^{Q}(D)\right)
.
\end{align*}
Observe that convexity of $\phi$ implies%
\[
\phi\left(  \frac{\alpha_{1}}{\alpha}\mu_{1}+\frac{\alpha_{2}}{\alpha}\mu
_{2}\right)  \leq\frac{\alpha_{1}}{\alpha}\phi(\mu_{1})+\frac{\alpha_{2}%
}{\alpha}\phi(\mu_{2})\text{ for }\mu_{1},\mu_{2}\in\lbrack0,1].
\]
This, together with $\phi(x)\leq0$, shows that for $E\in\mathfrak{A}_{n+1}%
^{K}$
\begin{align*}
&  \frac{\alpha_{1}}{\alpha}\phi\left(  \frac{\mu(D\cap\left(  \mathcal{U}%
\times K\right)  \cap E)}{\mu\left(  \mathcal{A}_{n+1}^{K}(D\cap\left(
\mathcal{U}\times K\right)  )\right)  }\right) \\
&  \geq\frac{\alpha_{1}}{\alpha}\phi\left(  \frac{\mu(D\cap\left(
\mathcal{U}\times K\right)  \cap E)}{\mu(\mathcal{A}_{n+1}^{K}(D\cap\left(
\mathcal{U}\times K\right)  ))}\right)  +\frac{\alpha_{2}}{\alpha}\phi\left(
\frac{\mu(D\cap\left(  \mathcal{U}\times\left(  Q\setminus K\right)  \right)
\cap E)}{\mu\left(  \mathcal{A}_{n+1}^{Q}(D)\setminus\mathcal{A}_{n+1}%
^{K}(D\cap\left(  \mathcal{U}\times K\right)  )\right)  }\right) \\
&  \geq\phi\left(  \frac{\alpha_{1}}{\alpha}\frac{\mu(D\cap\left(
\mathcal{U}\times K\right)  \cap E)}{\mu(\mathcal{A}_{n+1}^{K}(D\cap\left(
\mathcal{U}\times K\right)  ))}+\frac{\alpha_{2}}{\alpha}\frac{\mu
(D\cap\left(  \mathcal{U}\times\left(  Q\setminus K\right)  \right)  \cap
E)}{\mu\left(  \mathcal{A}_{n+1}^{Q}(D)\setminus\mathcal{A}_{n+1}^{K}%
(D\cap\left(  \mathcal{U}\times K\right)  )\right)  }\right)
\end{align*}%
\[
=\phi\left(  \frac{\mu(D\cap\left(  \mathcal{U}\times K\right)  \cap E)}%
{\mu\left(  \mathcal{A}_{n+1}^{Q}(D)\right)  }+\frac{\mu(D\cap\left(
\mathcal{U}\times\left(  Q\setminus K\right)  \right)  \cap E)}{\mu\left(
\mathcal{A}_{n+1}^{Q}(D)\right)  }\right)  =\phi\left(  \frac{\mu(D\cap
E)}{\mu(\mathcal{A}_{n+1}^{Q}(D)}\right)  .
\]
With (\ref{E1}) it follows that%
\[
-H_{\rho^{-n\tau}\mu}\left(  \mathfrak{A}_{n+1}^{K}\left\vert \mathfrak{A}%
_{n}^{K})\right.  \right)  \geq\rho^{-n\tau}\sum_{D\cap\left(  \mathcal{U}%
\times K\right)  \in\mathfrak{A}_{n}^{K}}\mu(\mathcal{A}_{n+1}^{Q}%
(D))\sum_{E\in\mathfrak{A}_{n+1}^{K}}\phi\left(  \frac{\mu(D\cap E)}%
{\mu(\mathcal{A}_{n+1}^{Q}(D)}\right)  .
\]
We may add further negative summands (corresponding to elements $D\in
\mathfrak{A}_{n}^{Q}\setminus\mathfrak{A}_{n}^{K}$ and $E\in\mathfrak{A}%
_{n+1}^{Q}\setminus\mathfrak{A}_{n+1}^{K}$) and estimate this by%
\[
\geq\rho^{-n\tau}\sum_{D\in\mathfrak{A}_{n}^{Q}}\mu\left(  \mathcal{A}%
_{n+1}^{Q}(D)\right)  \sum_{E\in\mathfrak{A}_{n+1}^{Q}}\phi\left(  \frac
{\mu(D\cap E)}{\mu(\mathcal{A}_{n+1}^{Q}(D)}\right)  =-H_{\rho^{-n\tau}\mu
}\left(  \mathfrak{A}_{n+1}^{Q}\left\vert \mathfrak{A}_{n}^{Q,n+1})\right.
\right)  .
\]
Next take the sum over the conditional entropies as specified in Definition
\ref{Def_cond}, divide by $n$ and take the limit for $n\rightarrow\infty$.
This shows, as claimed, that $h_{\mu}^{inc}(\mathcal{C}_{\tau}^{K})\leq
h_{\mu}^{inc}(\mathcal{C}_{\tau}^{Q})$. Now take the infimum over all
invariant $(Q,\eta)$-partitions $\mathcal{C}_{\tau}^{Q}$, the infimum over all
invariant $(K,\eta)$-partitions $\mathcal{C}_{\tau}^{K}\,$and, finally, let
$\tau\rightarrow\infty$. This concludes the proof.
\end{proof}

\section{Invariant $W$-control sets\label{Section4}}

In this section, we will describe certain subsets of complete approximate
controllability within a subset of the state space called control sets. They
will yield a relatively invariant set $K$ as considered in Theorems
\ref{TheoremA_alt} and Theorem \ref{TheoremA}.

For systems in discrete time, subsets of the state space where approximate or
exact controllability holds, have been analyzed in diverse settings. Relevant
contributions are due, in particular, to Albertini and Sontag \cite{AlbS91,
AlbS93}, Sontag and Wirth \cite{SonW98}, Wirth \cite{Wirt95, Wirt98NOLCOS} as
well as Patr\~{a}o and San Martin \cite{PatrS07} (there are subtle differences
in the definitions).

We recall the following notions and facts from the abstract framework in
\cite{PatrS07}, slightly modified for our purposes. A local semigroup
$\mathcal{S}$ on a topological space $X$ is a family of continuous maps
$\phi:\mathrm{dom}\phi\rightarrow X$ with open domain $\mathrm{dom}\phi\subset
X $ such that for all $\phi,\psi\in\mathcal{S}$ with $\psi^{-1}(\mathrm{dom}%
\phi)\not =\emptyset$ it follows that $\phi\circ\psi:\psi^{-1}(\mathrm{dom}%
\phi)\rightarrow X$ also is in $\mathcal{S}$.

For $x\in X$ the orbit is $\mathcal{S}x=\{\phi(x)\left\vert \phi\in
\mathcal{S}\text{ and }x\in\mathrm{dom}\phi\right.  \}$ and the backward orbit
is%
\begin{equation}
\mathcal{S}^{\ast}x=\{y\in X\left\vert \exists\phi\in\mathcal{S}%
:\phi(y)=x\right.  \}=\bigcup\nolimits_{\phi\in\mathcal{S}}\phi^{-1}(x).
\label{S*}%
\end{equation}
A local semigroup $\mathcal{S}$ is called accessible if $\mathrm{int}%
(\mathcal{S}x)\not =\emptyset$ and $\mathrm{int}(\mathcal{S}^{\ast}%
x)\not =\emptyset$ for all $x\in X$.

\begin{definition}
\label{Definition5.1}A control set for a local semigroup $\mathcal{S}$ on $X$
is a nonvoid subset $D\subset X$ such that (i) $y\in\mathrm{cl}(\mathcal{S}x)$
for all $x,y\in D$ (ii) for every $x\in D$ there are $\phi_{n}\in
\mathcal{S},n\in\mathbb{N}$, such that $\phi_{n}\circ\cdots\circ\phi_{1}(x)\in
D$ for all $n\in\mathbb{N}$ and (iii) the set $D$ is maximal with this property.

The transitivity set $D_{0}$ of $D$ is the set of all elements $x\in D$ such
that $x\in\mathrm{int}(\mathcal{S}^{\ast}x)$.
\end{definition}

\begin{remark}
Patr\~{a}o and San Martin \cite{PatrS07} define control sets in the following
slightly different way: A control set for $\mathcal{S}$ is a subset $D\subset
X$ such that $y\in\mathrm{cl}(\mathcal{S}x)$ for all $x,y\in D$, the set $D$
is maximal with this property, and there is $x\in D$ with $x\in\mathrm{int}%
(\mathcal{S}^{\ast}x)$. The latter condition means that the transitivity set
is nonvoid. Hence a control set as defined above with nonvoid transitivity set
is a control set in the sense of \cite{PatrS07}. Conversely, for an accessible
semigroup Patr\~{a}o and San Martin show that a control set $D$ in their sense
has the following properties:

The transitivity set $D_{0}$ is open and dense in $D$ and invariant in $D$,
i.e., $\mathcal{S}D_{0}\cap D\subset D_{0}$ (cf. \cite[Proposition
4.10]{PatrS07}) and, by \cite[Proposition 4.8]{PatrS07},%
\begin{equation}
D=\mathrm{cl}(\mathcal{S}x)\cap\mathcal{S}^{\ast}x\text{ for }x\in D_{0}.
\label{char1}%
\end{equation}
It follows that property (ii) in Definition \ref{Definition5.1} is satisfied.
In fact, this is clear for $x\in D_{0}$ and for an arbitrary point $x$ in $D$
there is $\phi\in\mathcal{S}$ with $\phi(x)\in D_{0}$, since $D_{0}$ is open.
Thus for an accessible semigroup the control sets as defined above with
nonvoid transitivity set coincide with the control sets in the sense of
\cite{PatrS07}.
\end{remark}

The following result is Patr\~{a}o and San Martin \cite[Proposition
4.15]{PatrS07}.

\begin{proposition}
\label{PS_Proposition 4.15}Let $D$ be a control set with nonvoid transitivity
set $D_{0}$ for an accessible local semigroup $\mathcal{S}$. Then the
transitivity set $D_{0}$ is open and dense in $D$ and the following statements
are equivalent:

(i) $\mathrm{cl}(\mathcal{S}x)\subset\mathrm{cl}D$ for all $x\in D$.

(ii) $D$ is closed and $\mathcal{S}$-invariant, i.e., $\mathcal{S}x\in D$ for
all $x\in D$.

(iii) $\mathrm{cl}D$ is $\mathcal{S}$-invariant.
\end{proposition}

Next we use these concepts in our context. Again we consider control system
(\ref{0.1}), but now we will restrict the state space to an open subset of the
state space $M$. For $\omega\in\Omega$ let $f_{\omega}:=f(\cdot,\omega
):M\rightarrow M$. Then the solutions $\varphi(k,x,u),u=(\omega_{i})$ can be
written in the form%
\[
\varphi(k,x,u)=f_{\omega_{k-1}}\circ\cdots\circ f_{\omega_{0}}(x).
\]
Let $W$ be an open nonvoid subset of the state space $M$. The maps $f_{\omega
},\omega\in\Omega$, generate the following family of continuous maps on $W$:
\[
f_{\omega}^{0}(x):=f_{\omega}(x):\mathrm{dom}f_{\omega}^{0}\rightarrow
W,\mathrm{dom}f_{\omega}^{0}:=\{x\in W\left\vert f(x,\omega)\in W\right.  \}
\]
and for $k\geq1$ and $u=(\omega_{0},\ldots,\omega_{k-1})\in\Omega^{k}$%
\begin{align*}
f_{u}^{k}(x)  &  :=f_{\omega_{k-1}}\circ\cdots\circ f_{\omega_{0}%
}(x):\mathrm{dom}f_{u}^{k}\rightarrow W,\\
\mathrm{dom}f_{u}^{k}  &  :=\{x\in W\left\vert f_{\omega_{i-1}}\circ
\cdots\circ f_{\omega_{0}}(x)\in W\text{ for }i=1,\ldots,k-1\right.  \}.
\end{align*}
Observe that the domain of $f_{u}^{k}$ is open and the maps $f_{\omega}$ are
continuous. Hence it follows that%
\begin{equation}
\mathcal{S}:=\left\{  f_{u}^{k}\left\vert k\in\mathbb{N},u\in\mathcal{U}%
\right.  \right\}  \label{local}%
\end{equation}
forms a local semigroup on $X:=W$.

For $x\in W,u\in\mathcal{U}$ and $k\in\mathbb{N}$ the corresponding solution
of (\ref{0.1}) in $W$ is denoted by $\varphi_{W}(k,x,u)$, if the solution of
(\ref{0.1}) satisfies $\varphi(i,x,u)\in W$ for $i=1,\ldots,k$. Thus
$\varphi_{W}(k,x,u)=f_{u}^{k}(x)$.

\begin{definition}
For $x\in M$ the $W$-reachable set $\mathbf{R}^{W}(x)$ and the $W$%
-controllable set $\mathbf{C}^{W}(x)$, resp., are%
\begin{align*}
\mathbf{R}^{W}(x)  &  :=\left\{  y\in W\left\vert \exists k\geq1~\exists
u\in\mathcal{U}:y=\varphi_{W}(k,x,u)\right.  \right\}  ,\\
\mathbf{C}^{W}(x)  &  :=\{y\in W\left\vert \exists u\in\mathcal{U}~\exists
k\geq1:\varphi_{W}(k,y,u)=x\right.  \}.
\end{align*}

\end{definition}

Note that in the language of local semigroups one has $\mathbf{R}%
^{W}(x)=\mathcal{S}x$ and $\mathbf{C}^{W}(x)=\mathcal{S}^{\ast}x$. Next we
specify maximal subsets of complete approximate controllability within $W$.

\begin{definition}
\label{Definition3.1}For system (\ref{0.1}) a subset $D\subset W$ is called a
$W$-control set if (i) $D\subset\mathrm{cl}_{W}\mathbf{R}^{W}(x)$ for all
$x\in D$, (ii) for every $x\in D$ there is $u\in\mathcal{U}$ with $\varphi
_{W}(k,x,u)\in D,k\in\mathbb{N}$, and (iii) $D$ is a maximal set with
properties (i) and (ii). A $W$-control set $D$ is called an invariant
$W$-control set if $\mathbf{R}^{W}(x)\subset\mathrm{cl}$$_{W}D$ for all $x\in
D $.
\end{definition}

Here the closures are taken with respect to $W$, and for an invariant
$W$-control set $\mathrm{cl}$$_{W}\mathbf{R}^{W}(x)=\mathrm{cl}$$_{W}D$ for
all $x\in D$. If $W=M$, we omit the index $W$ and just speak of control sets
and invariant control sets (if they have nonvoid interior they actually
coincide with the control sets and invariant control sets, respectively,
considered in Colonius, Homburg, Kliemann \cite{ColoHK10}). The transitivity
set $D_{0}$ of a $W$-control set $D$ is the set of all $x\in D$ with
$x\in\mathrm{int}\mathcal{S}^{\ast}x=\mathrm{int}\mathbf{C}^{W}(x)$.

It is immediate that the $W$-control sets coincide with the control sets for
the local semigroup $\mathcal{S}$ on $X=W$ defined in (\ref{local}).
Accessibility of the considered local semigroup $\mathcal{S}$ means that%
\begin{equation}
\mathrm{int}\mathbf{R}^{W}(x)\not =\emptyset\text{ and }\mathrm{int}%
\mathbf{C}^{W}(x)\not =\emptyset\text{ for all }x\in W. \label{access0}%
\end{equation}
This is certainly valid if%
\begin{equation}
\mathrm{int}f(x,\Omega)\cap W\not =\emptyset\text{ and }\mathrm{int}\{y\in
W\left\vert \exists\omega\in\Omega:f(y,\omega)=x\right.  \}\not =%
\emptyset\text{ for all }x\in W. \label{access1}%
\end{equation}
If accessibility holds, Proposition \ref{PS_Proposition 4.15} shows that a $W
$-control set $D$ with nonvoid transitivity set $D_{0}$ is an invariant
$W$-control set if and only if $\mathrm{cl}_{W}D=D$ and $\mathbf{R}%
^{W}(x)\subset D$ for all $x\in D$ if and only if $\mathbf{R}^{W}%
(x)\subset\mathrm{cl}_{W}D$ for all $x\in\mathrm{cl}_{W}D$. In particular, an
invariant $W$-control set is an invariant set in $W$.

It is also of interest to know when closedness in $W$ of a $W$-control set
already implies that it is an invariant $W$-control set. The proof of the
following proposition is adapted from Wirth \cite[Proposition 4.1.4]{Wirt95}.

\begin{proposition}
Suppose that control system (\ref{0.1}) satisfies accessibility condition
(\ref{access0}).

(i) Then every invariant $W$-control set $D$ is closed in $W$ and has nonvoid interior.

(ii) If for every $x\in W$ the set $\mathrm{cl}_{W}\left[  f(x,\Omega)\cap
W\right]  $ is path connected, then a $W$-control set $D$ which is closed in
$W $ and has nonvoid interior is an invariant $W$-control set.
\end{proposition}

\begin{proof}
(i) By the previous remarks the set $D$ is closed and has nonvoid interior
since $\emptyset\not =\mathrm{int}\mathbf{R}^{W}(y)\subset D$.

\qquad(ii) If $D=W$, there is nothing to prove. Otherwise we have to show for
every $x\in D$ that $\mathrm{cl}_{W}\mathbf{R}^{W}(x)\subset\mathrm{cl}_{W}D$
or, equivalently, that $\mathbf{R}^{W}(x)\subset D$, since $D$ is closed in
$W$. For every $y\in D$ there are $k\in\mathbb{N}$ and a control $u$ such that
$\varphi_{W}(k,y,u)\in\mathrm{int}D$. By continuous dependence on initial
values there exists an open neighborhood $V(y)$ of $y$ with $\varphi
_{W}(k,V(y),u)\subset\mathrm{int}D$. Taking the union of all $V(y)$ one finds
an open set $V\supset D$ such that for every $z\in V$ the intersection
$\mathbf{R}^{W}(z)\cap\mathrm{int}D\not =\emptyset$ and therefore
$D\subset\mathrm{cl}_{W}\mathbf{R}^{W}(z)$.

Assume now, contrary to the assertion, that there exist $x\in D$ and a control
value $\omega$ such that $f(x,\omega)\in W\setminus D$. As $D\subset
\mathrm{cl}_{W}\mathbf{R}^{W}(x)$ there exists $y\in\mathbf{R}^{W}(x)\cap D$
and hence there is $\omega\in\Omega$ with $f(x,\omega)\in D$, by maximality of
$W$-control sets.

We have shown that $\mathrm{cl}_{W}\left[  f(x,\Omega)\cap D\right]
\not =\emptyset$ and $\mathrm{cl}_{W}\left[  f(x,\Omega)\cap W\right]
\not \subset D$. Since $\mathrm{cl}_{W}\left[  f(x,\Omega)\cap W\right]  $ is
path connected by assumption, it follows that there exists $z\in
\mathrm{cl}_{W}\left[  f(x,\Omega)\cap W\right]  \cap\left(  V\setminus
D\right)  $. By continuity, this implies that $z\in\mathrm{cl}_{W}%
\mathbf{R}^{W}(y)$ for all $y\in D$ and, by construction of $V$, one has
$D\subset\mathrm{cl}_{W}\mathbf{R}^{W}(z)$ and thus $z\in D$ by the maximality
property of control sets. This is a contradiction.
\end{proof}

The following result constructs a set $K$ which is invariant in $Q$ and
satisfies the assumptions of Theorem \ref{TheoremA_alt}, hence it determines
the invariance entropy.

\begin{theorem}
\label{Theorem_B}Suppose that control system (\ref{0.1}) satisfies
accessibility condition (\ref{access0}) and let $Q\subset M$ be compact and
equal to the closure of its interior $W:=\mathrm{int}Q$. Furthermore, assume

(i) there are only finitely many invariant $W$-control sets $D_{1}%
,\ldots,D_{\ell}$ and their transitivity sets are nonvoid.

(ii) For every $x\in Q$ there is an invariant $W$-control set $D_{i}%
\subset\mathrm{cl}\mathbf{R}^{W}(x)$.

(iii) Let $K:=\bigcup_{i=1}^{\ell}\mathrm{cl}D_{i}$, suppose that this union
is disjoint and that $f(K,\Omega)\cap(\partial Q\setminus K)=\emptyset$.

Then every set $\mathrm{cl}D_{i}$ and hence $K$ is invariant in $Q$.

If $\eta$ is a quasi-stationary measure for $\nu$ on $\Omega$ and the maps
$f(\cdot,\omega),\omega\in\Omega$, are nonsingular with respect to $\eta$,
then the invariance entropies with respect to $\mu=\nu^{\mathbb{N}}\times\eta$
of $K$ and $Q$ coincide with%
\begin{equation}
\max_{i\in\{1,\ldots,\ell\}}h_{\mu}(\mathrm{cl}D_{i})\leq h_{\mu}(K)=h_{\mu
}(Q)\leq\sum_{i=1}^{\ell}h_{\mu}(\mathrm{cl}D_{i}). \label{ineq1}%
\end{equation}

\end{theorem}

\begin{proof}
The assumption implies that the boundaries of $Q$ and $W$ coincide, $\partial
Q=\partial W$. First we show that for every invariant $W$-control set $D$ its
closure $\mathrm{cl}D$ is invariant in the set $Q$. This also implies that $K$
is invariant in $Q$. By Proposition \ref{PS_Proposition 4.15} $D$ is closed in
$W$ and hence $\mathrm{cl}D=D\cup(\partial D\cap\partial Q)$. \ Suppose,
contrary to the assertion, that there are $x\in\mathrm{cl}D$ and $\omega
\in\Omega$ with $f(x,\omega)\in Q\setminus\mathrm{cl}D$.

If $x\in D$ then invariance in $W$ of $D=\mathrm{cl}_{W}D$ implies that
$\emptyset=f(x,\Omega)\cap(W\setminus D)=f(x,\Omega)\cap(W\setminus
\mathrm{cl}D)$. Since $\partial Q=\partial W$ it follows that $f(x,\omega
)\in\partial Q\setminus\mathrm{cl}C$. But assumption (iii) excludes this case.

It remains to discuss the case $x\in\partial D\cap\partial Q$. Then either
$f(x,\omega)\in W\setminus D$ or $f(x,\omega)\in\partial Q\setminus D$. In the
first case, Proposition \ref{PS_Proposition 4.15} implies $D=\mathrm{cl}%
(D_{0})$ and hence continuity of $f$ implies that there is $y\in D_{0}$ with
$f(y,\omega)\in W\setminus D$. This is excluded since $D$ is invariant in $W$.
The second case $f(x,\omega)\in\partial Q\setminus D$ is excluded by
assumption (iii) and it follows that $\mathrm{cl}D$ is invariant in $Q$.

Assumptions (i) and (ii) show that for every $x\in Q$ there are an invariant
$W$-control set $D_{i}$, a natural number $k_{0}$ and a control $u$ such that
$\varphi(k_{0},x,u)$ is in the transitivity set $D_{0}$ of $D_{i}$. Since
$D_{0}$ is open, continuity with respect to $x$ and compactness of $Q$ imply
that the assumptions of Theorem \ref{TheoremA_alt}(ii) are satisfied and it
follows that $h_{\mu}(Q)=h_{\mu}(K)$. The inequalities in (\ref{ineq1}) follow
by invariance of $\mathrm{cl}D_{i}$ in $Q$ from Theorem \ref{TheoremA_alt}(iii).
\end{proof}

Next we discuss when the assumptions of Theorem \ref{Theorem_B} are satisfied.
The following theorem characterizes the existence of finitely many invariant
$W$-control sets.

\begin{theorem}
\label{Theorem_rel_inv}Consider a control system of the form (\ref{0.1})
satisfying accessibility condition (\ref{access0}).

(i) Let $x\in Q$ and assume that there exists a compact set $F\subset W$ such
that for all $y\in\mathbf{R}^{W}(x)$ one has $\mathrm{cl}\mathbf{R}^{W}(y)\cap
F\not =\emptyset$. Then there exists an invariant $W$-control set
$D\subset\mathrm{cl}_{W}\mathbf{R}^{W}(x)$.

(ii) If there is a compact set $F\subset W$ such that $\mathrm{cl}%
\mathbf{R}^{W}(x)\cap F\not =\emptyset$ for all $x\in Q$, then for every $x\in
Q$ there is an invariant $W$-control set $D$ with $D\subset\mathrm{cl}%
\mathbf{R}^{W}(x)$ and there are only finitely many invariant $W$-control sets
$D_{1},\ldots,D_{\ell}$.

(iii) Conversely, suppose that for every $x\in Q$ there is an invariant
$W$-control set $D$ with $D\subset\mathrm{cl}\mathbf{R}^{W}(x)$ and there are
only finitely many invariant $W$-control sets $D_{1},\ldots,D_{\ell}$ and they
all have nonvoid transitivity set. Then there is a compact set $F\subset W$
such that $\mathbf{R}^{W}(x)\cap F\not =\emptyset$ for all $x\in Q$.
\end{theorem}

\begin{proof}
(i) For $y\in\mathbf{R}^{W}(x)$ let $F(y):=\mathrm{cl}\mathbf{R}^{W}(y)\cap
F$. Consider the family of nonvoid and compact subsets of $W$ given by
$\mathcal{F}=\{F(y)\left\vert y\in F(x)\right.  \}$. Then $\mathcal{F}$ is
ordered via%
\[
F(y)\preccurlyeq F(z)\text{ if }z\in\mathrm{cl}\mathbf{R}^{W}(y).
\]
Every linearly ordered subset $\{F(y_{i}),i\in I\}$ has an upper bound
\[
F(y)=\bigcap\nolimits_{i\in I}F(y_{i})\text{ for some }y\in\bigcap
\nolimits_{i\in I}F(y_{i}),
\]
since the intersection of decreasing compact subsets of the compact set $F$ is
nonvoid. Thus Zorn's lemma implies that the family $\mathcal{F}$ has a maximal
element $F(y)$. Now the set%
\[
D:=\mathrm{cl}_{W}\mathbf{R}^{W}(y)
\]
is an invariant $W$-control set: Note first that by condition (\ref{access0})
the set $D$ has nonvoid interior. Every $z\in D$ is in $\mathrm{cl}%
_{W}\mathbf{R}^{W}(y)$ and, conversely, $y\in\mathrm{cl}_{W}\mathbf{R}^{W}(z)$
for every $z\in D$ since otherwise $y\not \in F(z)=\mathrm{cl}\mathbf{R}%
^{W}(z)\cap F\subset\mathrm{cl}_{W}\mathbf{R}^{W}(y)\cap F=F(y)$, hence
$F(y)\preceq F(z)$ and $F(y)\not =F(z)$ contradicting the maximality of $F(y)$.

Continuity implies that for all $z_{1},z_{2}\in\mathrm{cl}_{W}\mathbf{R}%
^{W}(y)$ one has $z_{2}\in\mathrm{cl}_{W}\mathbf{R}^{W}(z_{1})$, hence
$D=\mathrm{cl}_{W}\mathbf{R}^{W}(y)$ is a $W$-control set. It is an invariant
$D $-control set since for $z\in D=\mathrm{cl}_{W}\mathbf{R}^{W}(y)$
continuity implies $\mathbf{R}^{W}(z)\subset\mathrm{cl}_{W}\mathbf{R}%
^{W}(y)=D$.

(ii) Let $x\in Q$. Then, by (i), one finds an invariant $W$-control set in
$\mathrm{cl}\mathbf{R}^{W}(x)$. Suppose that there are countably many pairwise
different invariant $W$-control sets $D_{n},n\in\mathbb{N}$. Thus
$\mathrm{cl}\mathbf{R}^{W}(y)=\mathrm{cl}D_{n}\subset\mathrm{cl}_{W}D_{n}%
\cup\partial Q$ for all $y\in D_{n}$. Since $D_{n}$ is closed in
$W=\mathrm{int}Q$ and $F\subset W$, the property $\mathrm{cl}\mathbf{R}%
^{W}(y)\cap F\not =\varnothing$ implies that $D_{n}\cap F\not =\emptyset$.

There are points $y_{n}\in D_{n}\cap F$ converging to some point $y\in F$. By
part (i) one finds an invariant $W$-control set $D$ contained in
$\mathrm{cl}_{W}\mathbf{R}^{W}(y)$, and hence there is a point $z$ in the
intersection of $\mathbf{R}^{W}(y)$ and the interior of $D$. Now continuity
implies that for every $n$ large enough there is a point $z_{n}$ in $D_{n}$
with $\mathbf{R}^{W}(z_{n})\cap\mathrm{int}D\not =\varnothing$. This
contradicts invariance in $W$ of the $W$-control sets $D_{n}$.

(iii) Choose for each of the finitely many invariant $W$-control sets
$D_{i},i=1,\ldots,\ell$, a point $x_{i}\in D_{i}$ and define $F:=\{x_{1}%
,\ldots,x_{\ell}\}$. Let $x\in Q$. By assumption, there is an invariant
$W$-control set $D_{i}$ contained in $\mathrm{cl}\mathbf{R}^{W}(x)$. Hence
there is a point in the intersection of $\mathbf{R}^{W}(x)$ and the
transitivity set of $D_{i}$, thus (\ref{char1}) implies that $x_{i}%
\in\mathbf{R}^{W}(x)$ and the assertion follows.
\end{proof}

It remains to discuss when an invariant $W$-control set has a nonvoid
transitivity set. Obviously, this holds if for the local semigroup
$\mathcal{S}$ one has that $\mathcal{S}x$ and $\mathcal{S}^{\ast}x$ are open,
cf. San Martin and Patr\~{a}o \cite[Corollary 5.4]{PatrS07} for a situation
where this occurs. Instead of this strong assumption, we will require
smoothness of $f$ and use Sard's Theorem as well as some arguments from Wirth
\cite{Wirt98NOLCOS}.

Consider a control system of the form%
\begin{equation}
x_{k+1}=f(x_{k},u_{k}),k\in\mathbb{N}=\{0,1,\ldots\}, \label{W_2}%
\end{equation}
under the following assumptions: The state space $M$ is a $C^{\infty}$-
manifold of dimension $d$ endowed with a corresponding metric. The set of
control values $\Omega\subset\mathbb{R}^{m}$ satisfies $\Omega\subset
\mathrm{cl}(\mathrm{int}\Omega)$. Let $\tilde{\Omega}$ be an open set
containing $\mathrm{cl}\Omega$. The map $f:M\times\tilde{\Omega}\rightarrow M$
is a $C^{\infty}$-map and $W\subset M$ is a nonvoid open subset.

We define for $k\geq1$ a $C^{\infty}$-map%
\[
F_{k}:W\times\mathrm{int}\Omega^{k}\rightarrow W,F_{k}(x,u):=\varphi
_{W}(k,x,u).
\]
The domain of $F_{k}$ is an open subset of $W\times\Omega^{k}$.

A pair $(x,u)\in W\times\mathrm{int}\Omega^{k}$ is called regular, if
$\mathrm{rank}\frac{\partial F_{k}}{\partial u}(x,u)=d$ (clearly, this implies
$mk\geq d$). For $x\in M$ and $k\in\mathbb{N}$ the regular $W$-reachability
set and the regular $W$-controllability set, resp., are
\begin{align*}
\mathbf{\hat{R}}_{k}^{W}(x)  &  :=\left\{  y\in W\left\vert \exists
u\in\mathrm{int}\Omega^{k}:y=\varphi_{W}(k,x,u)\text{ and }(x,u)\text{ is
regular}\right.  \right\}  ,\\
\mathbf{\hat{C}}_{k}^{W}(x)  &  :=\{y\in W\left\vert \exists u\in
\mathrm{int}\Omega^{k}:\varphi_{W}(k,y,u)=x\text{ and }(y,u)\text{ is
regular}\right.  \},
\end{align*}
and the regular $W$-reachability set $\mathbf{\hat{R}}^{W}(x)$ and
$W$-controllability set $\mathbf{\hat{C}}^{W}(x)$ are given by the respective
union over all $k\in\mathbb{N}$. It is not difficult to see that
$\mathbf{\hat{R}}^{W}(x)$ and $\mathbf{\hat{C}}^{W}(x)$ are open for every $x$
(cf. Wirth \cite[Lemma 8]{Wirt98NOLCOS}). In the notation of the local
semigroup $\mathcal{S}$ defined in (\ref{local}) one has%
\[
\mathbf{\hat{C}}^{W}(x)\subset\mathrm{int}\mathbf{C}^{W}(x)=\mathrm{int}%
\mathcal{S}^{\ast}x.
\]
In order to show that the transitivity set of a control set is nonvoid, we
start with the following observations. If $\mathbf{\hat{R}}_{k_{0}}%
^{W}(x)\not =\varnothing$ it follows that $\mathbf{\hat{R}}_{k}^{W}%
(x)\not =\varnothing$ for all $k>k_{0}$. Accessibility condition
(\ref{access0}) implies for all $x\in W$ that there is $k_{0}\in\mathbb{N}$
such that for all $k\geq k_{0}$ one has $\mathrm{int}\mathbf{R}_{k}%
^{W}(x)\not =\varnothing$ and
\[
\mathbf{R}_{k}^{W}(x)\subset\mathrm{cl}\{y=\varphi_{W}(k,x,u)\in
\mathrm{int}\mathbf{R}_{k}^{W}(x)\left\vert u\in\mathrm{int}\Omega^{k}\right.
\}.
\]
Sard's Theorem (cf., e.g., Katok and Hasselblatt \cite[Theorem A.3.13]%
{KatH95}) implies that the set of points $\varphi_{W}(k,x,u)\in\mathbf{R}%
_{k}^{W}(x)$ such that $(x,u)$ is not regular has Lebesgue measure zero.

The following proposition presents conditions which imply that the
transitivity set of a control set is nonvoid.

\begin{proposition}
Consider system (\ref{W_2}) and assume that accessibility condition
(\ref{access0}) holds. Then for every $W$-control set $D\subset W$ with
nonvoid interior the transitivity set $D_{0}$ is nonvoid.
\end{proposition}

\begin{proof}
Let $x\in\mathrm{int}D$ and consider an open neighborhood $V_{1}\subset D$ of
$x$. There is $k_{0}\in\mathbb{N}$ such that the reachable set $\mathbf{R}%
_{k}^{W}(x)$ at time $k$ has nonvoid interior for all $k\geq k_{0}$. There are
$k\geq k_{0}$ and $\varphi_{W}(k,x,u)\in\mathbf{R}_{k}^{W}(x)\cap
\mathrm{int}D$, hence we may assume that there is $y:=\varphi_{W}%
(k,x,u)\in\mathrm{int}\mathbf{R}_{k}^{W}(x)\cap\mathrm{int}D$. Then, by Sard's
Theorem, it follows that there is a point $y=\varphi_{W}(k,x,u)\in
\mathrm{int}D$ with regular $(x,u)$, i.e., $y\in\mathrm{int}D\cap
\mathbf{\hat{R}}_{k}^{W}(x)$. Then $x\in\mathbf{\hat{C}}_{k}(y)\subset
\mathrm{int}\mathbf{C}(y)$. Let $V\subset\mathrm{int}\mathbf{C}(y)$ be a
neighborhood of $x$. Then $x\in D\subset\mathrm{cl}\mathbf{R}^{W}(y)$, hence
there is $z\in V\cap\mathbf{R}^{W}(y)\subset D$ and thus $y\in\mathbf{C}(z)$.

By construction, the point $z\in D$ satisfies $z\in\mathrm{int}\mathbf{C}%
(y)\subset\mathrm{int}\mathbf{C}(z)$, hence it is in the transitivity set of
$D$.
\end{proof}

\begin{remark}
Wirth \cite{Wirt98NOLCOS} defines (for $W=M$) the regular core of a control
set $D$ denoted by $\mathrm{core}(D)$ as the set of points in $D$ for which
the regular reachability and controllability sets intersect $D$ and shows,
under real analyticity assumptions, that $\mathrm{core}(D)$ is open and dense
in $\mathrm{cl}D$. Thus $\mathrm{core}(D)\subset D_{0}$. This generalizes
earlier results by Albertini and Sontag \cite[Section 3]{AlbS91} and
\cite[Section 7]{AlbS93} (again for $W=M$). They define the core of a control
set as%
\begin{equation}
\{x\in\mathrm{int}D\left\vert \mathrm{int}\mathbf{R}(x)\cap D\not =%
\varnothing\text{ and }\mathrm{int}\mathbf{C}(x)\cap D\not =\varnothing
\right.  \}, \label{core_AS}%
\end{equation}
and assume, in particular, that for every $\omega\in\Omega$ the map
$f_{\omega}:=f(\cdot,\omega)$ is a global diffeomorphism on $M$ and that for
all $x\in M$ the set of points which can be reached by finite compositions of
maps of the form $f_{\omega}$ and $f_{\omega}^{-1}$ applied to $x$ coincides
with $M$. They show that the core is open and that it is dense in $D$. Again,
it is clear that the transitivity set of $D$ is contained in the core as
defined in (\ref{core_AS}).
\end{remark}

The results above give conditions which imply that the metric invariance
entropy for a quasi-stationary measure is determined by the invariant
$W$-control sets. It may be of interest to analyze the relations between the
supports of quasi-stationary measures and $W$-control sets. For ergodic
stationary measures $\eta$, Colonius, Homburg and Kliemann \cite[Lemma
5]{ColoHK10} shows (for certain random diffeomorphisms) that the support of
$\eta$ coincides with an invariant control set. The following proposition
gives a result in that direction.

In the setting of (\ref{quasi}) define the $k$-step transition function
$p_{k}^{Q}(x,A),x\in Q,A\subset Q$,%
\[
p_{1}^{Q}(x,A):=p(x,A),p_{k}^{Q}(x,A):=\int_{Q}p_{k-1}^{Q}(y,A)p(x,dy),k>1.
\]
Then (cf. Colonius \cite[Remark 2.8]{Colo16a}) it follows for all $k\geq1$ and
all $A\subset Q$ that%
\begin{equation}
\rho^{k}\eta(A)=\int_{Q}p_{k}^{Q}(z,A)\eta(dz). \label{quasi_k}%
\end{equation}

\begin{proposition}
\label{Proposition_support}Consider control system (\ref{0.1}) and let
$Q\subset M$ be equal to the closure of its interior $W:=\mathrm{int}Q$.
Consider an invariant $W$-control set $D$ with nonvoid transitivity set
$D_{0}$ and let $\eta$ be a quasi-stationary measure for $\nu$ on $\Omega$.
Assume that for every $x\in W,k\geq1$ and $y\in\mathbf{R}_{k}^{W}(x)$ every
neighborhood $V(y)$ satisfies $p_{k}(x,V(y))>0.$

Then every quasi-stationary measure $\eta$ with $\mathrm{supp}\eta\cap
D\not =\varnothing$ satisfies $D\subset\mathrm{supp}\eta$.
\end{proposition}

\begin{proof}
Suppose, contrary to the assertion, that there is $y\in D\setminus
\mathrm{supp}\eta$. Since $W\setminus\mathrm{supp}\eta$ is open in $W$ there
is a neighborhood $V(y)$ of $y$ in $W$ such that $V(y)\cap\mathrm{supp}%
\eta=\varnothing$. By Proposition \ref{PS_Proposition 4.15} the transitivity
set $D_{0}$ is dense in $D$, hence we may take $y\in D_{0}$. By assumption,
there is $x\in\mathrm{supp}\eta\cap D$. Thus $x\in\mathbf{R}_{k}^{W}(y)$ for
some $k\geq1$. By continuity, there is a neighborhood $V(x)$ such that
$\mathbf{R}_{k}^{W}(z)\cap V(y)\not =\varnothing$ for all $z\in V(x)$ and, by
the definition of the support, one has $\eta(V(x))>0$. The assumption
guarantees that $p_{k}^{Q}(z,V(y))>0$ for all $z\in V(x)$. This contradicts
the quasi-stationarity property (\ref{quasi_k}), since $\eta(D\setminus
\mathrm{supp}\eta)=0$, while%
\[
\int_{Q}p_{k}^{Q}(z,D\setminus\mathrm{supp}\eta)\eta(dz)\geq\int_{V(x)}%
p_{k}^{Q}(z,V(y))\eta(dz)>0\text{.}%
\]

\end{proof}

The results above show that the metric invariance entropy of a subset $Q$ of
the state space is already determined on a subset $K$ that can be
characterized using controllability properties. For the topological invariance
entropy of systems in continuous time, an analogous result has been shown in
Colonius and Lettau \cite[Theorem 5.2]{ColoL16}.

Finally, we present two examples illustrating $W$-control sets and their
relation to invariance entropy. First we take a closer look at Example
\ref{Example1} in order to discuss $W$-control sets in a simple situation.

\begin{example}
Recall that $f_{\alpha}:{\mathbb{R}}/{\mathbb{Z}}\times\lbrack-1,1]\rightarrow
{\mathbb{R}}/{\mathbb{Z}}$ is given by
\[
f_{\alpha}(x,\omega)=x+\sigma\cos(2\pi x)+A\omega+\alpha\mod 1.
\]
For $\alpha\leq\alpha_{0}$ there is an invariant control set $\hat{D}^{\alpha
}=[d(\alpha),e(\alpha)]\not ={\mathbb{R}}/{\mathbb{Z}}$ that varies
continuously with $\alpha$. For $\alpha>\alpha_{0}$ the only control set is
the invariant control set $\hat{D}^{\alpha}={\mathbb{R}}/{\mathbb{Z}}$.

Now consider $\alpha>\alpha_{0}$ and $W=(0.2,0.5)$. There is a unique
invariant $W$-control set $D^{\alpha}$, which has the form $D^{\alpha
}=[d(\alpha),0.5)$. Then one easily sees that the invariant $W$-control sets
$D^{\alpha}$ are closed in $W$ and their transitivity sets are nonvoid.
Furthermore, the closure $\mathrm{cl}D^{\alpha}=[d(\alpha),0.5]$ is invariant
in $Q:=\mathrm{cl}W=[0,2,0.5]$.

For the uniform distribution $\nu$ on $\Omega:=[-1,1]$, \cite[Theorem
3]{ColoHK10} implies that for all $\alpha>0$ there is a unique stationary
measure $\hat{\eta}^{\alpha}$ satisfying $\eta(B)=\int_{{\mathbb{R}%
}/{\mathbb{Z}}}p(x,B)\eta(dx)$ for all $B\subset{\mathbb{R}}/{\mathbb{Z}}$. It
has support equal to the invariant control set $\hat{D}^{\alpha}$. For
$\alpha>\alpha_{0}$ Theorem \ref{Theorem_B} shows that for every
quasi-stationary measure $\eta^{\alpha}$ of $Q=\mathrm{cl}W$ the invariance
entropy of $Q$ coincides with the invariance entropy of the closure
$\mathrm{cl}D^{\alpha}=[d(\alpha),0.5]$ of the invariant $W$-control set
$D^{\alpha}$ as already seen in Example \ref{Example1}. From Proposition
\ref{Proposition_support} we obtain the additional information that the
quasi-stationary measure $\eta_{\alpha}$ has support equal to $\mathrm{cl}%
D^{\alpha}=[d(\alpha),0.5]$.
\end{example}

We modify this example, so that in addition to an invariant $W$-control set
$D_{2}^{\alpha}$ there is there is a second $W$-control set $D_{1}^{\alpha}$
(to the left of $D_{2}^{\alpha}$) which is not invariant.

\begin{example}
\label{Example2}Define $f_{\alpha}:{\mathbb{R}}/{\mathbb{Z}}\times
\lbrack-1,1]\rightarrow{\mathbb{R}}/{\mathbb{Z}}$ by%
\begin{equation}
f_{\alpha}(x,\omega)=x+\sigma\cos(4\pi x)+A\omega+\alpha\mod 1. \label{Ex2}%
\end{equation}
Here take $Q=[0,1,0.7]$ and with $W=(0.1,0.7)$ the invariant $W$-control set
is $D_{2}^{\alpha}=[d(\alpha),0.7)$ (to the right). The $W$-control set
$D_{1}^{\alpha}=[a(\alpha),b(\alpha))$ (to the left) is not invariant in $W$,
since exit to the right is possible. Invoking Theorem \ref{Theorem_B} one sees
that for every quasi-stationary measure $\eta^{\alpha}$ of $Q=\mathrm{cl}W$
(such that $f_{\alpha}(\cdot,\omega),\omega\in\lbrack-1,1]$, are nonsingular)
the invariance entropy for $\mu^{\alpha}$ of $Q$ coincides with the invariance
entropy of $\mathrm{cl}D_{2}^{\alpha}$. As above, for the uniform distribution
$\nu$ on $\Omega$ there is no stationary measure with support contained in $Q$.
\end{example}

The many open problems in this area include the following\textbf{. }When is
the support of a quasi-stationary measure contained in the (closure of) the
union of the invariant $W$-control sets? In this situation, Theorems
\ref{TheoremA_alt} and \ref{TheoremA} would hold trivially, since, naturally,
the metric invariance entropy is determined on the support of the
quasi-stationary measure (cf. also Remark \ref{vanDoornPollett}). Theorem
\ref{Theorem_B} reduces the analysis of the invariance entropy from arbitrary
closed sets $Q$ to invariant $W$-control sets. Hence their measure theoretic
invariance entropy is of particular interest. For control sets, the
topological invariance entropy has been characterized in Kawan \cite{Kawa11b}
and da Silva and Kawan \cite{daSilKawa16a} using hyperbolicity and Lyapunov exponents.

\end{document}